\documentclass [10 pt, letterpaper,openany]{article}

  \usepackage[latin1]{inputenc}
 \usepackage[english]{babel}
 \usepackage{amsfonts}
 \usepackage{amsmath}
 \usepackage{amssymb} 
 \usepackage{latexsym}
 \usepackage{amsthm}
 \usepackage{ragged2e}
 \usepackage{graphicx}
 \usepackage{calc}
\newlength{\depthofsumsign}
 \graphicspath{{./figs/}}
 \usepackage{vmargin}
\usepackage{pgfplots,tikz}
\usepackage{setspace}
\usepackage{stackengine}
\setstackgap{S}{1.3pt}
\usepackage{float}
\usepackage{enumerate} 
\usepackage{relsize}
\usetikzlibrary{arrows}

\usepackage [ all ]{xy}
\usepackage[titletoc]{appendix}
\usepackage{hyperref}
\numberwithin{equation}{section}
\usepackage{empheq}
\hypersetup{
    colorlinks=true,
    linkcolor=blue,
    filecolor=magenta,      
    urlcolor=cyan,
    bookmarks,
    citecolor=red,
}
\usepackage[nottoc,notlot,notlof]{tocbibind}
\usepackage{pdfpages}
\usepackage{wallpaper}
\usepackage{enumitem}
\makeatletter

\DeclareSymbolFont{symb}{OMS}{lmsy}{m}{n}
\DeclareMathSymbol{\emptysett}{\mathord}{symb}{"3B}
\renewcommand{\emptyset}{\emptysett}

\theoremstyle{plain}
\newtheorem{theorem}{Theorem}[section] 

\theoremstyle{definition}
\newtheorem{prop}[theorem]{Proposition}
\newtheorem{lemma}[theorem]{Lemma}
\newtheorem{definition}[theorem]{Definition}
\newtheorem{rem}[theorem]{Remark}

\newcommand{\ds}{\displaystyle}

\newcommand{\N}{\ensuremath{\mathbb N }}

\newcommand{\R}{\ensuremath{\mathbb R }}

\newcommand{\C}{\ensuremath{\mathbb{C} }}

\newcommand{\sn}[2]{\left\|#1\right\|_{#2}}
\newcommand{\lsim}{\ensuremath\mathrel{\raisebox{1pt}{\stackunder{$<$}{\rotatebox{-0}{\resizebox{9pt}{2pt}{$\boldsymbol{\sim}$}}}}}}
\newcommand{\Imag}{\operatorname{Im}}
\newcommand{\rmi}{\mathrm{i}}
\newcommand{\rmd}{\mathrm{d}}
\newcommand{\bx}{\mathbf{x}}
\newcommand{\bk}{\mathbf{k}}
\newcommand{\bF}{\mathbf{F}}
\newcommand{\bG}{\mathbf{G}}
\newcommand{\bE}{\mathbf{E}}
\newcommand{\bU}{\mathbf{U}}
\newcommand{\bH}{\mathbf{H}}
\newcommand{\bP}{\mathbf{P}}
\newcommand{\bM}{\mathbf{M}}
\newcommand{\bD}{\mathbf{D}}
\newcommand{\bB}{\mathbf{B}}
\newcommand{\be}{\mathbf{e}}
\newcommand{\bh}{\mathbf{h}}
\newcommand{\bpsi}{\mathbf{\psi}}
\newcommand{\bphi}{\mathbf{\phi}}
\newcommand{\bp}{\mathbf{p}}
\newcommand{\bm}{\mathbf{m}}
\newcommand{\bbA}{\mathbb{A}}
\newcommand{\curl}{\operatorname{curl}}

\renewcommand{\emptyset}{\emptysett}

\begin{document}
\title{Long time behaviour of the solution of Maxwell's equations in dissipative generalized Lorentz materials (I) A  frequency dependent Lyapunov function approach}

\author{Maxence Cassier$^{a}$, Patrick Joly$^{b}$ and Luis Alejandro Rosas Mart\'inez$^{b}$  \\ \ \\
{
\footnotesize $^a$ Aix Marseille Univ, CNRS, Centrale Marseille, Institut Fresnel, Marseille, France }\\ 
{\footnotesize $^b$ ENSTA / POEMS$^1$, 32 Boulevard Victor, 75015 Paris, France}\\ 
{\footnotesize (maxence.cassier@fresnel.fr, patrick.joly@inria.fr, alejandro.rosas@ensta-paris.fr)}}
\footnotetext[1]{POEMS (Propagation d'Ondes: Etude Math\'ematique et Simulation) is a mixed research team (UMR 7231) between CNRS (Centre National de la Recherche Scientifique), ENSTA Paris (Ecole Nationale Sup\'erieure de Techniques Avanc\'ees) and INRIA (Institut National de Recherche en Informatique et en Automatique).}

	\maketitle
	\begin{abstract}

It is well-known that electromagnetic dispersive structures such as metamaterials can be modelled by generalized Drude-Lorentz models. The present paper is the first  of two articles dedicated to dissipative generalized Drude-Lorentz open structures. We wish to quantify the loss in such media in terms of the long time decay rate of the electromagnetic energy for the corresponding Cauchy problem. By using an approach based on frequency dependent Lyapounov estimates,  we show that this decay is polynomial in time.  These results extend to an unbounded structure the ones obtained  for bounded media in \cite{Nicaise2020} via a quite different method based on the notion of cumulated past history and semi-group theory. A great advantage of the approach developed here is to be less abstract and directly connected to the physics of the system via energy balances.
\end{abstract}	

{\noindent \bf Keywords:} Maxwell's equations, passive electromagnetic media,  dissipative generalized Lorentz models, long time electromagnetic energy decay rate, frequency-dependent  Lyapunov   estimates.

\section{Introduction and motivation} 
The study of the long time behaviour of solutions of dispersive and dissipative models for linear wave propagation has already been extensively studied in the literature, primarily for applications in visco-elasticity and more recently in electromagnetism. The subject has recently known a regain of interest related to metamaterials. We can refer for instance, in electromagnetism, to the article \cite{cas-jol-kach-17} in which we presented a systematic construction of mathematical models compatible with physically motivated principle such as causality and passivity (see also \cite{GralakTip,cas-mil-17,Welters}).  The common point to all these models lies in that the constitutive laws include memory effects corresponding to time convolution nonlocal operators that induce dispersion (the velocity of waves is frequency dependent) and dissipation (the energy decay of the solution) that are in often intimately related.
\\ [6pt]
For such models one of the most natural question is the study of the long time behaviour of the corresponding Cauchy problem: prove that the energy of the  solution tends to 0 when $t$ tends to $+ \infty$ and study the rate of decay.  This is of course closely related to the theory of control and stabilization of dynamical systems where one commonly distinguishes the notion of exponential stability (which corresponds to an exponential decay of the energy) and polynomial stability (the energy decays as the inverse of a positive power of $t$).
\subsection{Maxwell's equations in dispersive media}
\subsubsection{General features}
Maxwell's equations relate the electric and magnetic inductions $\textbf{D}(\bx,t)$ and $\textbf{B}(\bx,t)$ ($\bx\in \mathbb{R}^3$ and $t>0$ are respectively the  space and variables)  to the the electric and magnetic fields $\textbf{E}(\bx,t)$ and $\textbf{H}(\bx,t)$: 
\begin{equation} \label{maxwell3D}
\left\{ \begin{array}{l}
\partial_t\,\textbf{D}-\nabla\times\textbf{H} =0,\\ [8pt]
\partial_t\,\textbf{B}+\nabla\times\textbf{E}=0.
\end{array} \right.
\end{equation}
On the other hand, one defines the electric polarization and magnetization by 
\begin{equation} \label{CL1}
\left\{ \begin{array}{ll}
\textbf{D} = \varepsilon_0 \, \textbf{E} + \textbf{P}_{\mathrm{tot}}, &  \quad \textbf{P}_\mathrm{tot} : \mbox{electric polarization}, \\ [8pt]
\textbf{B} = \mu_0 \, \textbf{E} + \textbf{M}_{\mathrm{tot}}, &  \quad \textbf{M}_{\mathrm{tot}} : \mbox{magnetization} .
\end{array} \right.
\end{equation}
where $\varepsilon_0>0$ and $\mu_0>0$ are the vacuum permittivity 
and  permeability.
The above equations are completed by the following non local constitutive laws (we consider the case of a homogeneous medium)
\begin{equation} \label{CL2}
\left\{ \begin{array}{ll}
\ds \textbf{P}_{\mathrm{tot}}(\cdot, t) = \varepsilon_0 \, \int_0^t \chi_e(t-s) \; \textbf{E}(\cdot, s) \; \rmd s, \\ [12pt]
\ds \textbf{M}_{\mathrm{tot}}(\cdot, t) =\mu_0 \,  \int_0^t \chi_m(t-s) \; \textbf{H}(\cdot, s) \; \rmd s, 
\end{array} \right.
\end{equation}
where $\chi_e$ and $\chi_m$ are the electrical and magnetic susceptibilities of the material
(convolutions products being understood in the distributional sence, see for e.g. \cite{cess-96,Zem-72}, for $(\chi_e, \chi_m)$  not in $L^1$).\\ [12pt] 
In the Fourier-Laplace domain 
$$
\textbf{E}(\cdot, t) \quad \longrightarrow \quad \hat{\textbf{E}}(\cdot, \omega) =  \int_0^{+\infty} \textbf{E}(\cdot, t) \; e^{\rmi \omega t} \; \rmd t, \quad \Imag\, \omega > 0,
$$
\eqref{CL1} and \eqref{CL2} reduce to 
\begin{equation} \label{CLomega}
\left\{ \begin{array}{ll}
\hat{\textbf{D}}(\cdot, \omega) = \varepsilon(\omega)  \, \hat{\textbf{E}}(\cdot, \omega),\\ [8pt]
\hat{\textbf{B}}(\cdot, \omega) = \mu(\omega)  \, \hat{\textbf{H}}(\cdot, \omega),
\end{array} \right.
\end{equation}
where the complex permittivity $\varepsilon(\omega)  $ and the complex permeability  $\mu(\omega)$  are given in terms of the Fourier-Laplace transform of the susceptibility functions:
\begin{equation}\label{eq.complexperm}
\varepsilon(\omega)  = \varepsilon_0 \, \big( 1 + \hat \chi_e (\omega)\big) \quad \mbox{ and } \quad  \mu(\omega)  = \mu_0 \, \big( 1 + \hat \chi_m (\omega)\big).
\end{equation}
where $\varepsilon(\omega)\to \varepsilon_0 $ \mbox{  and } $\mu(\omega)\to \mu_0 $ when $\omega\to \infty$ in $\C^+:= \{ \omega \in \C \, / \, \Imag \, \omega > 0 \big\}$. In other words, the material behaves as the vacuum at high frequencies.
In the frequency domain, passivity, causality and the high frequency behaviour are traduced by the fact that (see \cite{Ber-11,cas-jol-kach-17,cas-mil-17,Welters,Zem-72} for more details)
\begin{equation}\label{eq.mathpass}
\mbox{$\omega \mapsto \omega \, \varepsilon(\omega)$ and $\omega \mapsto \omega \, \mu(\omega)$ 
	are {\it Herglotz functions,}}
\end{equation} that is to say analytic functions from $\C^+$ into its closure $\overline{\C^+}$. Furthermore as  the susceptibilities $\chi_{e}$ and $\chi_m$ are real-valued functions in the time domain,  the permittivity and permeability satisfy $\overline{\varepsilon(\omega)} = \varepsilon(- \overline \omega)$ and $\overline{\mu(\omega)} = \mu(- \overline \omega), \ \forall \; \omega \in \C^+$.
\begin{rem}\label{rem_passivity} [About the notion of passivity] The condition \eqref{eq.mathpass} is the condition which is most often used to define passive materials: we called it {\it mathematical passivity} in \cite{cas-jol-kach-17}. In the same article, we define the related notion {\it physical passivity} which is associated to the Cauchy problem associated to (\ref{maxwell3D}., \ref{CL1}, \ref{CL2}), seen as an evolution problem with respect to the electromagnetic field $(\mathbf{E}, \mathbf{H}$). In other words, we look at the free evolution of the system i.e in the absence of external sources. More precisely a material is  {\it physical passive} if and only if the electromagnetic energy 
		\begin{equation} \label{def_energyEM0}
		\mathcal{E}(t)\equiv \mathcal{E}(\mathbf{E}, \mathbf{H}, t):=\frac{1}{2}\Big(\varepsilon_0\, \int_{\R^3} |{\mathbf{E(\bx,t)}}|^2  \rmd \bx+\mu_0\, \int_{\R^3}|{\mathbf{H(\bx,t)}}|^2 \rmd \bx\Big).
		\end{equation}
can never exceeds its value at $t=0$, that is to say
		\begin{equation} \label{def_physic alpassivity}
\forall \; t \geq 0, \quad \mathcal{E}(t) \leq \mathcal{E}(0).
\end{equation}
It is emphasized in \cite{cas-jol-kach-17} that the above property does not mean that the electromagnetic energy is a decreasing function of time. \end{rem} 
\subsubsection{The generalized Lorentz media}
In this paper, we shall concentrate of the most well-known subclass of models: the (dissipative) generalized Lorentz media. Such model will be called {\it local} because of the relationship between ${\bf D}$ and ${\bf E}$ or ${\bf B}$ and ${\bf H}$ can be written with ordinary differential equations. More precisely, these correspond to
\begin{equation} \label{decomposition}
\textbf{P}_{\mathrm{tot}} = \varepsilon_0 \, \sum_{j=1}^{N_e} \Omega_{e,j}^2 \; \textbf{P}_j , \quad \textbf{M}_{\mathrm{tot}}  = \mu_0 \, \sum_{\ell=1}^{N_m} \Omega_{m,\ell}^2 \; \textbf{M}_\ell, 
\end{equation}
where each $\textbf{P}_j$ (resp. each $\textbf{M}_\ell$) is related to ${\bf E}$ (resp. ${\bf H}$) by an ordinary differential equation 
\begin{equation} \label{ODE_Lorentz}
\left\{ 
\begin{array}{ll}
\partial_t^2\,\mathbf{P}_j+\alpha_{e,j}\,\partial_t\,\mathbf{P}_j+\omega_{e,j}^2\,\mathbf{P}_j=\mathbf{E}, & 1 \leq j \leq N_e, \\[12pt]
\partial_t^2\,\mathbf{M}_\ell+\alpha_{m,\ell}\,\partial_t\,\mathbf{M}_\ell+\omega_{m,\ell}^2\,\mathbf{M}_\ell=\mathbf{H}, & 1 \leq \ell \leq N_m,
\end{array}
\right.
\end{equation}
completed by $0$ initial conditions
\begin{equation} \label{IS_Lorentz}
\left\{ 
\begin{array}{ll}
\mathbf{P}_j (\cdot,0)= \partial_t\,\mathbf{P}_j (\cdot,0) =0, & 1 \leq j \leq N_e,\\[12pt]
\mathbf{M}_\ell (\cdot,0)= \partial_t\,\mathbf{M}_\ell (\cdot,0) =0, & 1 \leq \ell \leq N_m.
\end{array}
\right.
\end{equation}
In the above equations, the (real) coefficients $\Omega_{e,j},\Omega_{m,\ell}, \omega_{e,j},\omega_{m,\ell}$ are supposed to satisfy
\begin{equation} \label{hypomega}
\Omega_{e,j} > 0,  \quad \omega_{e,j} \geq 0, \quad 1 \leq j \leq N_e,  \quad  \Omega_{m,\ell} > 0, \quad  \omega_{m,\ell} \geq 0,\quad 1 \leq \ell \leq N_m \; ,
\end{equation}
while for stability/dissipation issues the coefficients $(\alpha_{m,\ell}, \alpha_{m,\ell})$ must be positive 
\begin{equation} \label{hypalpha}
\alpha_{e,j} \geq 0, \quad 1 \leq j \leq N_e,  \quad \alpha_{m,\ell} \geq 0\quad 1 \leq \ell \leq N_m.
\end{equation}
Moreover, the reader will easily check that one can assume without any loss of generality that the couples $(\alpha_{e,j}, \omega_{e,j})$ (resp. $(\alpha_{m,\ell}, \omega_{m,\ell})$) are all distinct the ones from the others. \\ [12pt] 
Note that (\ref{decomposition}, \ref{ODE_Lorentz}) corresponds to 
\begin{equation} \label{epsmuLorentz}
(a) \quad \varepsilon(\omega) = \varepsilon_0 \, \Big( 1 - \sum_{j=1}^{N_e} \frac{\Omega_{e,j}^2}{\omega^2 + i \, \alpha_{e,j} \, \omega - \omega_{e,j}^2}\Big) ,  \quad (b) \quad \mu(\omega) = \mu_0 \, \Big( 1 - \sum_{\ell=1}^{N_m} \frac{\Omega_{m,\ell}^2}{\omega^2 + i \, \alpha_{m,\ell} \, \omega - \omega_{m,\ell}^2}\Big).  
\end{equation}
Straightforward calculations show that (\ref{decomposition}, \ref{ODE_Lorentz}) are equivalent to \eqref{CL2} with
\begin{equation} \label{Lorentz-Kernels}
\chi_e = \sum_{j=1}^{N_e} \; \Omega_{e,j}^2 \; \chi_{e,j}, \quad \chi_m = \sum_{\ell=1}^{N_m} \Omega_{m,l}^2 \; \chi_{m,\ell}
\end{equation}
where the expression of each $\chi_{\nu,j}$ for  $\nu=e,m$ and $j\in \{ 1,\ldots, N_\nu\}$ is given by
		\begin{equation} \label{expchie} 
		 \begin{array}{lll}
		(i) & 	\chi_{\nu,j}(t) =2 \, \delta_{\nu,j}^{-1} \; \sinh \big(\delta_{\nu,j} \, t/2\big) \;  e^{- ({\alpha_{\nu,j} \,t}/{2}) }, \;  & \mbox {if }  \alpha_{\nu,j} > 2 \, \omega_{\nu,j}, \\ [8pt] 
		(ii) & 	\chi_{\nu,j}(t) = 2 \, \delta_{\nu,j}^{-1} \; \sin \big(\delta_{\nu,j} \, t/2 \big) \;  e^{- ({\alpha_{\nu,j} \,t}/{2}) }, \;  & \mbox {if }  \alpha_{\nu,j} < 2 \, \omega_{\nu,j},\\ [8pt] 
		(iii) & 		\chi_{\nu,j}(t) =t \;   e^{- (\alpha_{\nu,j} \,t/2)}, & \mbox {if }  \alpha_{\nu,j} = 2 \, \omega_{\nu,j},
			\end{array} \end{equation} 
where we have set 
			$\delta_{\nu,j} = \sqrt {\alpha_{\nu,j}^2 - 4 \, \omega_{\nu,j}^2} \ \mbox{ if } \alpha_{\nu,j} \geq 2 \, \omega_{\nu,j}  \mbox{ and }  \delta_{\nu,j} = \sqrt {4 \, \omega_{\nu,j}^2 - \alpha_{\nu,j}^2 }  \ \mbox{ if }  \alpha_{\nu,j} < 2 \, \omega_{\nu,j}.$ \\ [12pt]
	Note that each kernel {$\chi_{\nu,j}$ is not monotonous with respect to time as soon as $\omega_{\nu,j} > 0$ or $\alpha_{\nu,j}> 0$ and tends to $0$ when $t \rightarrow + \infty$ if (and only if) $\alpha_{\nu,j} > 0$ (see figure \ref{fig_noyaux}, first two pictures). As a consequence $\chi_\nu$ does not tend to $0$ at infinity as soon as one of the $\alpha_{\nu,j}$ vanishes (see figure \ref{fig_noyaux}, third picture).}
		\begin{figure}[h!]\label{fig_noyaux}	
		\begin{center}
			\includegraphics[scale=0.1875]{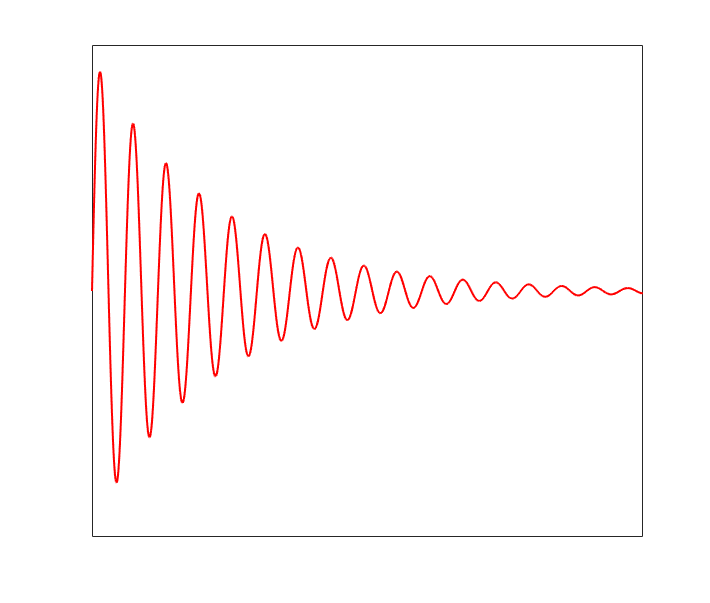}   \includegraphics[scale=0.19]{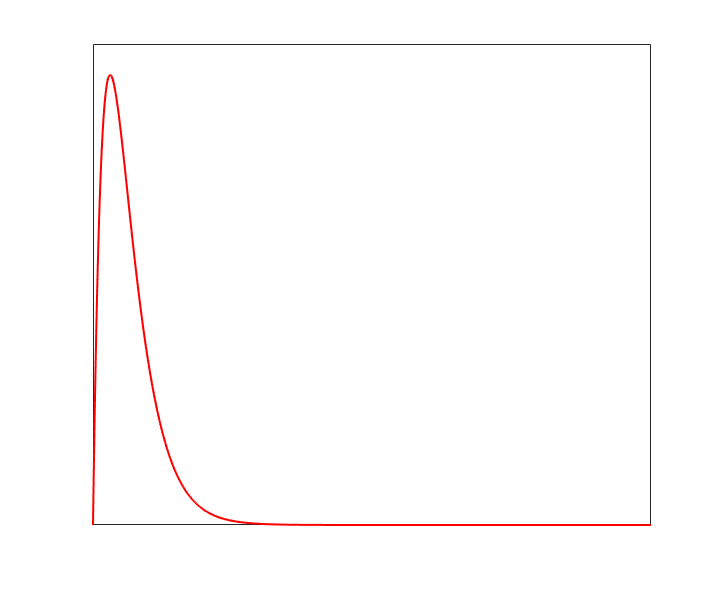} \includegraphics[scale=0.2]{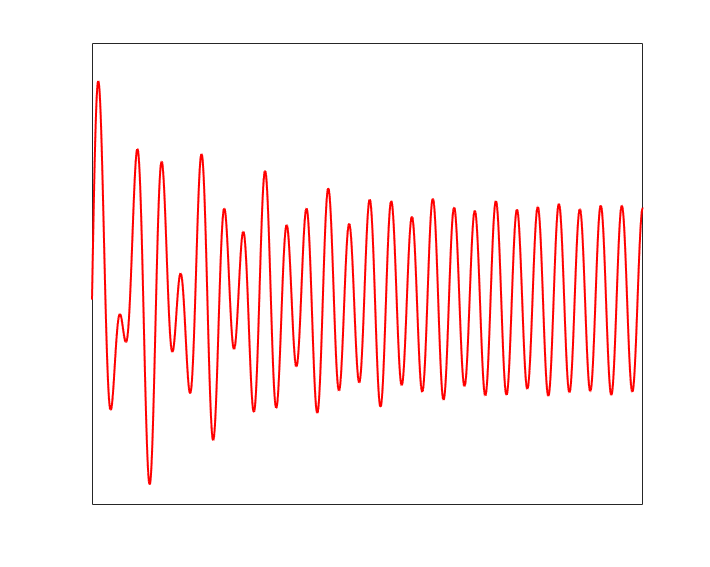}
		\end{center}
		\caption{Kernels as functions of time. Left: $\chi_{\nu,j}$  for  $\alpha_{\nu,j} > 2 \, \omega_{\nu,j} >0$. Center: $\chi_{\nu,j}$ for  $0 < \alpha_{\nu,j} < 2 \, \omega_{\nu,j}$. Right: $\chi_{\nu}$  for  $N_\nu = 2, \alpha_{\nu,1} = 0,  \, \alpha_{\nu,2} > 0$.}
	\end{figure}

\subsection{A brief review of the literature}
As said in introduction, there are already many existing results on the long time behaviour of the solution of dissipative dispersive systems. In this paragraph, we discuss in some detail some of the most {significant} contributions that are in close connection with the present work. \\ [12pt]
In the article \cite{Figotin}, the authors considered a very abstract evolution model
that in particular includes \eqref{CL2} with $\chi_m = 0$ and a function  $\chi_e$ for which $\omega \, \hat{\chi}_e(\omega)$  is a  Herglotz  function.
If one assumes that this function satisfies  the additional assumption that
 \begin{equation} \label{limits} \mbox{for a. e. }\omega \in \R, \quad\gamma_e (\omega)  := \lim_{\zeta \in \C^+\rightarrow \, \omega} \; \Imag \zeta \, \hat{\chi}_e (\zeta) \mbox{ exists,}
 \end{equation} 
 (which is satisfied by most of dispersive materials in physics and in particular by  generalized Lorentz materials) then 
 the  sufficient dissipation  condition (6.4) of \cite{Figotin}  is equivalent to 
\begin{equation} \label{dissFigotin}
 \mbox{for a. e. }\omega \in \R, \quad \gamma_e(\omega) > 0, \quad \gamma_e^{-1}\in L^1_{loc}(\R).
\end{equation}
When applied to the generalized Lorentz model {(see appendix \ref{sec-append-figotin})}, namely when $\varepsilon(\omega)$ is given by \eqref{epsmuLorentz}
the above condition corresponds to 
\begin{equation} \label{dissFigotinLorentz}
\exists \; 1 \leq j \leq N_e \mbox{ such that } \omega_{e,j} = 0 \mbox{ and } \alpha_{e,j} > 0 .
\end{equation}

\noindent Under this condition, it is proven that the electromagnetic energy (see definition \ref{def_energyEM0})
tends to $0$ for any initial data $(\mathbf{E}_0, \mathbf{H}_0)$ in $L^2(\R^3)^3 \times L^2(\R^3)^3$:
\begin{equation} \label{FigotinDecay}
\forall \; (\mathbf{E}_0, \mathbf{H}_0) \in L^2(\R^3)^3 \times L^2(\R^3)^3, \quad 
\lim_{t\rightarrow + \infty} \mathcal{E}(t)= 0.
\end{equation}
This result is proven in \cite{cas-jol-kach-17} (section 4.4) in a much more pedestrian way on a toy problem  corresponding the Drude model with $N_e=N_m = 1$, $\omega_{e,1}=\omega_{m,1}=0$ and $\alpha_{e,1}, \alpha_{e,m} >0$.  \\ [10pt]
In the above references, the question of the rate of convergence to $0$ of the electromagnetic energy is not discussed. 
 This question is addressed in a series of work by S. Nicaise and her collaborator C. Pignotti  \cite{Nicaise2012}, \cite{Nicaise2020} (which generalizes \cite{Nicaise2012}) (see also \cite{Nicaise2021} for local dissipation models).
 These works consider the initial value problem in a bounded domain $\Omega \subset \R^3$ with perfectly conducting boundary conditions 
 for which they prove {\it polynomial stability} in the sense mentioned below (see estimate \eqref{NicaiseEstimates}).
The conditions for polynomial stability in \cite{Nicaise2020} are two fold:
\begin{itemize}
	\item[(i)]
The first condition is expressed in the time domain, more precisely in terms of regularity and decay properties of the kernels ${\chi}_e$ and ${\chi}_m$:
\begin{equation} \label{DecayNicaise}
\chi = {\chi}_e \mbox{ or } \chi_m \mbox{ satisfies } \chi \in C^2(\R^+), \quad 
\lim_{t \rightarrow + \infty} \chi'(t) = 0, \quad |\chi''(t)| \leq C \; e^{-\delta t} \ \mbox{(with } C , \delta > 0.)
\end{equation}
Note that \eqref{DecayNicaise} implies {that it exists $C_1>0$ such that $|\chi'(t)| \leq C_1 \, \delta^{-1} e^{-\delta t} $  and  $ |\chi(t)| \leq C_1$  for $t\geq 0$ (see e.g.  the appendix of \cite{Nicaise2020} for the  details). Hence, it follows (by integrations by parts of the Fourier-Laplace integral) that
\begin{equation}\label{eq.regulnicaise}
 \omega \,   \widehat{\chi}(\omega)=\rmi \, \widehat{\chi'}(\omega)+ 
\rmi \, \chi(0) \quad \mbox{ and }  \quad  \omega \,   \widehat{\chi}(\omega)=-\big(\widehat{\chi''}(\omega)+\chi'(0)\big) \, \omega^{-1}+ \rmi \, \chi(0),  \quad \forall  \; \omega\in \C^+ 
\end{equation}
where, as the Laplace-Fourier transform of a $L^1$ causal function, $\omega \mapsto  \widehat{\chi'}(\omega)$ and $\omega \mapsto  \widehat{\chi''}(\omega)$ are analytic on $\C^+$, continuous on $\overline{\C}^+$ and decay to $0$  when  $\omega\to \infty $ in  $\overline{\C^+}$.
Thus, using \eqref{eq.regulnicaise}, one observes that   \eqref{DecayNicaise} implies that the functions  $\omega \mapsto \omega \, \hat{\chi}_e(\omega)$ and $ \omega \mapsto \omega \, \hat{\chi}_m(\omega)$ are analytic on $\C^+$ and can be extended as  continuous and bounded functions in the closed upper half-plane $\overline {\C^+}$. Furthermore,  $\hat{\chi}=\hat{\chi}_e, \hat{\chi}_m$, one has  $\omega\hat{\chi}(\omega)= \rmi \chi(0) - \chi'(0) \, \omega^{-1}+ o(\omega^{-1})$ when $\omega\to \infty $ in  $\overline{\C^+}$. 
\\
	\item[(ii)] The second condition  is expressed in the frequency domain for real frequencies. It also has two parts. The first one is a strict positivity condition 
	\begin{equation} \label{Passivity-Dissipation-Nicaise}
{\forall \, \omega \in \R^*, \quad \gamma_e(\omega) =: \Imag \, \omega \, \hat{\chi}_e (\omega) > 0, \quad \gamma_m(\omega) :=\Imag \, \omega \, \hat{\chi}_m (\omega) >  0.}
\end{equation}
which is completed the additional assumption
\begin{equation} \label{Passivity-Dissipation-Nicaise-HF}
\exists \; \omega_0, q, C > 0 \mbox{ such that } \quad  |\omega|\geq \omega_0 \quad \implies \quad \gamma_e(\omega) \geq  C \, |\omega|^{-q},  \ \gamma_m(\omega) \geq   C \, \omega^{-q},
\end{equation}
that means the the (strictly positive) functions $\gamma_e(\omega)$ and $\gamma_e(\omega)$ do not decay too fast at infinity.}
\end{itemize}
Under assumptions \eqref{DecayNicaise}, \eqref{Passivity-Dissipation-Nicaise} and \eqref{Passivity-Dissipation-Nicaise-HF}, the authors of  \cite{Nicaise2020}  prove, for $H^1$ initial data, decay estimates of the form
\begin{equation} \label{NicaiseEstimates}
\mathcal{E}(t) \leq C_q \; t^{-\frac{2}{q}} \; \big( \|\bE_0\|_{H^1(\Omega)}^2+ \|\bH_0\|_{H^1(\Omega)}^2\big).
\end{equation}
When specified to the case of the Lorentz kernels, it is easy to see (with \eqref{Lorentz-Kernels} and \eqref{expchie}) that, in addition to the non-negativity of the coefficients $(\alpha_{e,j}, \alpha_{m, \ell})$, the conditions \eqref{DecayNicaise},  \eqref{Passivity-Dissipation-Nicaise} and \eqref{Passivity-Dissipation-Nicaise-HF} (which is then satisified for $q=2$) correspond  to the following strong dissipation condition. 
\begin{definition} [Strong Dissipation for Lorentz models] \label{StrongDissipation}
\begin{equation} \label{SD} \forall \; 1 \leq j \leq N_e, \quad \alpha_{e,j} >0, \quad  
\forall \; 1 \leq \ell \leq N_m, \quad \alpha_{m,\ell} >0. \end{equation} 
\end{definition}
\noindent It is worthwhile mentioning briefly the techniques of proof for \eqref{FigotinDecay} in \cite{Figotin} and \eqref{NicaiseEstimates} in \cite{Nicaise2020}. \\ [12pt]
The technique used in \cite{Figotin} is difficult to describe in a few lines but we can give some of the main ideas. The authors use an augmented formulation of the evolution problem where, typically, the convolutions \eqref{CL2} are hidden behind the introduction of new unknowns. In the very abstract framework of \cite{Figotin}, these new unknowns can be seen as elements of an adequately constructed Hilbert space. In the case of the Maxwell's equations, these new unknowns are functions of $(\bx,t)$ but also of additional variable $\xi$ that {varies in $\R$ (or a subset of $\R$), see \cite{Figotin} and also \cite{cas-jol-kach-17,GralakTip}}. The fundamental property of the obtained ``augmented" system is that it is conservative: in other words, it is an evolution problem associated with a self-adjoint operator $\mathcal{A}_c$ in the augmented Hilbert space:
$$
\frac{d \mathcal{U}_c}{dt} + \rmi \, \mathcal{A}_c \, \mathcal{U}_c= 0.
$$
This allows to use tools from spectral theory of self-adjoint operators and to obtain an adequate (Fourier-like)  integral representation of the solution of the original problem. The
convergence result \eqref{FigotinDecay} then appears as a consequence of the spectral theorem and the use of  Riemann-Lebesgue theorem on compact sets, that is justified by the assumption \eqref{dissFigotin}. \\ [12pt]
Although quite different, the approach of \cite{Nicaise2020} also starts from another augmented formulation of the original system in which the convolutions \eqref{CL2} are again replaced by additional  purely differential equations. The construction of this
augmented model relies of the very nice trick of Dafermos \cite{Daf-70} for treating viscoelasticity. This implies to introduce 
an additional time variable $s$ (that plays a similar role than  $\xi$ in \cite{cas-jol-kach-17}) and additional unknowns: the so called cumulated past histories of the fields ${\bf E}$ and ${\bf H}$. The convolution operators are then replaced by non homogeneous transport equations in the $(s,t)$ plane. Contrary to \cite{Figotin}, the augmented system is not conservative and is written as an autonomous evolution problem involving an operator $\mathcal{A}_a$ which is not self-adjoint but maximal dissipative: 
$$
\frac{d \mathcal{U}_a}{dt} + \rmi \, \mathcal{A}_a \, \mathcal{U}_a = 0
$$
This problem arises from the application of the theory of semi-groups. In particular, estimates \eqref{NicaiseEstimates} are obtained by the application of  theorem 2.4  of  \cite{Bor-10}. Applying this theorem essentially 
 requires to establish localization results for the spectrum of $\mathcal{A}_a$ (inside $\C^-=\{ \omega \in \C \mid \operatorname{Im}(\omega)<0\}$) and appropriate estimates for its resolvent, using the conditions (\ref{eq.regulnicaise},\ref{Passivity-Dissipation-Nicaise},\ref{Passivity-Dissipation-Nicaise-HF}).
 \subsection{About Lyapunov techniques}
To conclude this short bibliographical review, it is worthwhile mentioning that other stability results for dispersive/dissipative results have also been obtained via the Lyapunov technique:  roughly speaking, the idea is to
 derive some differential inequality (in time) for a certain functional of the solution, namely the Lyapunov funtion $\mathcal{L}$,
 which dominates the energy (or some function of the energy). The stability estimates are then obtained from the time integration of the differential inequality.
 \\ [12pt]
 In the context of dissipative systems with memory, this type of technique was introduced in particular to show {exponential stability} in the theory of linear viscoelasticity \cite{FabrizioMorro87} (see also \cite{FabrizioMorro97,Riv-04} in the context of electromagnetism) and used  in \cite{Riv-04}  where the authors establish polynomial stability estimates associated to various damping phenomena in electromagnetism to model, for instance, a rigid electric conductor or the ionized atmosphere (see also remark \ref{rem_passivity2}). However their technique cannot be applied to our case because it requires sign properties of the derivatives of convolution kernels, which clearly prevents from time oscillations as in \eqref{Lorentz-Kernels}. To give an idea about why this kind of assumption appears, let us come back  the system (\ref{maxwell3D}, \ref{CL1}, \ref{CL2}) in the whole space $\R^3$, provided that the kernels $\chi_e$ and $\chi_m$ are of class $C^3$ on $\R^+$. Adapting the techniques of  \cite{FabrizioMorro87} {to dispersive Maxwell's equations}, one can show (formally) the following identity (see appendix \ref{sec-app-Lya}) 
 \begin{equation} \label{Identitegenerale}
\frac{d}{dt} \mathcal{L}({\bf E}, {\bf H})(t) +  \mathcal{D}({\bf E}, {\bf H})(t) = 0
\end{equation}
where the Lyapunov function $\mathcal{L}({\bf E}, {\bf H})$ is a kind of augmented energy, namely
\begin{equation} \label{Lyapunovgeneral0}
\mathcal{L}({\bf E}, {\bf H})(t) = \mathcal{E}({\bf E}, {\bf H})(t)  + \mathcal{E}_{ad}({\bf E}, {\bf H})(t)
\end{equation} 
with the additional energy
\begin{equation} \label{Lyapunovgeneral}
\left| \begin{array}{lll}
\mathcal{E}_{ad}({\bf E}, {\bf H})(t) & = & \ds \frac{\varepsilon_0}{2} \, \chi_e'(t) \int_{\R^3} |\bE_p(\bx,t)|^2 \,  \rmd \bx + \frac{\mu_0 }{2} \, \chi_m'(t) \textbf{}\int_{\R^3} |\bH_p(\bx,t)|^2 \,  \rmd \bx  \\ [12pt]
& -  & \ds \frac{\varepsilon_0}{2} \int_{0}^{t} \chi_e''(t-s)  \Big(  \int_{\R^3} |\bE_p(\bx,t) - \bE_p(\bx,s) |^2 \,  \rmd \bx \Big) \, \rmd s \\ [12pt]
& -  & \ds\frac{\mu_0 }{2} \int_{0}^{t} \chi_m''(t-s) \Big(  \int_{\R^3} |\bH_p(\bx,t) - \bH_p(\bx,s) |^2 \,  \rmd \bx \Big) \, \rmd s. 
\end{array} \right. 
\end{equation} 
~\\ and the functional $\mathcal{D}({\bf E}, {\bf H}) $ is given by 
\begin{equation} \label{Dissipationfunction}
\left| \begin{array}{lll}
\mathcal{D}({\bf E}, {\bf H})(t) & = & \ds \varepsilon_0 \, \chi_e(0) \int_{\R^3}|\bE(\bx,t)|^2 \,  \rmd \bx + \mu_0  \, \chi_m(0) \int_{\R^3} |\bH(\bx,t)|^2 \,  \rmd \bx  \\ [12pt]
&-& \ds \frac{\varepsilon_0}{2} \, \chi_e''(t) \int_{\R^3} |\bE_p(\bx,t)|^2 \,  \rmd \bx - \frac{\mu_0 }{2} \, \chi_m''(t) \textbf{}\int_{\R^3} |\bH_p(\bx,t)|^2 \,  \rmd \bx  \\ [12pt]
& +  & \ds \frac{\varepsilon_0}{2} \int_{0}^{t} \chi_e'''(t-s)  \Big(  \int_{\R^3} |\bE_p(\bx,t) - \bE_p(\bx,s) |^2 \,  \rmd \bx \Big) \, \rmd s \\ [12pt]
& +  & \ds\frac{\mu_0 }{2} \int_{0}^{t} \chi_m'''(t-s) \Big(  \int_{\R^3} |\bH_p(\bx,t) - \bH_p(\bx,s) |^2 \,  \rmd \bx \Big) \, \rmd s.
\end{array} \right. 
\end{equation} 
 Sufficient conditions to ensure stability estimates and time decay results simply amount to check that
\begin{equation} \label{condition}
\mathcal{L}({\bf E}, {\bf H}) \geq \mathcal{E}({\bf E}, {\bf H}) \quad ( \; \Longleftrightarrow \;   \mathcal{E}_{ad}({\bf E}, {\bf H}) \geq 0)\quad \mbox{and} \quad \mathcal{D}({\bf E}, {\bf H}) \geq 0.
\end{equation}
It is clear on (\ref{Lyapunovgeneral}) and (\ref{Dissipationfunction}) that (\ref{condition}) lead to the (sufficient) conditions
\begin{equation} \label{conditionsepsmu}
\chi_\nu(0)\geq 0, \; \;  \chi_\nu'(t)\geq 0, \; \;  \chi_\nu''(t)\leq 0 \; \mbox{ and } \; \chi_\nu^{(3)}(t)\geq 0  \quad \forall \; t \in \mathbb{R}^+, \quad \nu = e,m.
\end{equation}
Such conditions are fulfilled for instance  if one  has
\begin{equation}\label{eq.stab}
\chi_{\nu}(0)>0, \ - \, \chi_\nu''(t) \geq \beta_\nu \; \chi_\nu'(t)\quad \mbox{and} \quad \chi_\nu'''(t) \geq - \, \beta_\nu \; \chi_\nu''(t) 
\end{equation}
which yields the existence of a constant $\delta > 0$ such that $\mathcal{D}({\bf E}, {\bf H}) \geq \delta \; 
\mathcal{L}({\bf E}, {\bf H}) $ so that immediately implies the {\it exponential stability} of Maxwell's equations in the sense that
\begin{equation} \label{ExponentialStability}
\mathcal{E}(t) \leq C \;e^{-\delta t}.
\end{equation}
An elementary example of susceptibility kernels satisfying \eqref{eq.stab}  is given by $\chi_{\nu}(t)=2-\mathrm{e}^{-t}$. 
\\ [12pt]
\noindent Unfortunately,  the conditions \eqref{condition} are useless for analysing  
the stability  of Maxwell's generalized Lorentz materials given by \eqref{ODE_Lorentz} and \eqref{hypomega},  the conditions \eqref{conditionsepsmu} are only satisfied  by non-dissipative Drude materials for which $\chi_e(t)= \Omega_e^2 \, t$ and $\chi_m(t)=  \Omega_m^2  \, t$.
\begin{rem} \label{rem_passivity2}
It is worthwhile to come back here to what we said in the remark \ref{rem_passivity} and more precisely on the possible equivalence (generally conjectured) between the two notions of {\it mathematical passivity} and {\it physical passivity}. 
What we show above is that the conditions \eqref{conditionsepsmu} are sufficient conditions for physical passivity. As a consequence, finding functions $\chi_e$ and $\chi_m$ satisfying \eqref{conditionsepsmu}  but such that the Herglotz property 
\eqref{eq.mathpass} would not hold, would provide a counter example to the equivalence. \\ [12pt]
The conditions \eqref{conditionsepsmu} are clearly reminiscent of the notion of Bernstein functions \cite{Berg-08,Schi-10}, i.e. positive continuous function $f: [0,\infty[ \to (0, \infty)$, $C^{\infty}$ on  $(0,+\infty)$ and whose derivative 
$f'$ is a completely monotonous functions, which means that the sign of the successive derivatives of $f$ alternate with the order of derivation:
\begin{equation} \label{bernstein}
\forall \; j \in \N\setminus\{0\} \quad (-1)^{j+1} \; f^{(j)}(t) \geq 0.
\end{equation}
Completely monotone functions are also characterized as the Laplace transforms of positive Borel measures on $[0, + \infty)$ (see, e.g. \cite{Berg-08,Schi-10}). It is known (see \cite{Hany-08}, Theorem 3.2 and corollary 3.14) that Bernstein functions satisfy the Herglotz property
\begin{equation} \label{herglotz_prop}
\omega \; \hat f(\omega) \mbox{ is an Herglotz  function, where  }  \hat f(\omega) \mbox{ is the Laplace-Fourier transform of } f.
\end{equation}
However, when the alternating sign property \eqref{bernstein} is only true for $j \leq 3$, which corresponds to \eqref{conditionsepsmu}, \eqref{herglotz_prop} could a priori fail.
\end{rem}

\subsection{Objectives and outline of our work} 
We revisit in this series of two papers is to revisit the stability theory of (\ref{maxwell3D}, \ref{CL1}, \ref{decomposition}, \ref{ODE_Lorentz}), that is to say Maxwell's equations in generalized dissipative Lorentz media.
For the simplicity of exposition, we shall consider the problem posed in the whole space $\R^3$ that authorizes the use of the Fourier transform in space. We have a double objective 
	\begin{itemize} 
		\item propose new constructive proofs of stability estimates based on elementary tools that avoids any use of "black box" results of  abstract mathematical theory,
		\item extend the existing results with less restrictive assumptions than those appearing in \cite{Nicaise2020} or \cite{Nicaise2021}, namely to the case of the weak dissipativity assumption \eqref{WD}. 
			\end{itemize}
		\noindent 
		\begin{definition} [Weak Dissipation for Lorentz models] \label{WeakDissipation}
			\begin{equation} \label{WD} 
			\sum_{j=1}^{N_e} \alpha_{e,j}  + \sum_{\ell=1}^{N_m} \alpha_{m,\ell} > 0.
			\end{equation}
		\end{definition} \noindent
	In this first paper, we shall restrict ourselves to the strong dissipation assumption \eqref{SD} and wish to recover in a quite explicit form the results from \cite{Nicaise2020} with a technique inspired by the Lyapunov approach. Compared to more standard Lyapunov methods, {we introduce} frequency dependent Lyapunov functions (understand  spatial frequency or wave numbers) that in particular allows to distinguish the respective roles of low and high frequencies 
	(the role of low frequencies, that does not appear in a bounded domain as in \cite{Nicaise2020} is due to the fact that we work in an unbounded domain). \\ [12pt]
	In the second paper, we shall use, in addition to the Fourier transform, a spectral representation of the solution that will permit us to derive sharp asymptotic long time estimates. This approach is less tricky than the frequency dependent Lyapunov approach but is technically more involved because non self-adjoint operators have to be handled. It also has the interest to only assume the weak dissipativity condition \eqref{WD} and to provide optimal results.
	We also think that, in both papers, the arguments we shall use are quite close to physical  notions (plane waves, dispersion analysis, energy balance)  which should make these papers more accessible to physicists. \\ [12pt]
The outline of the present paper is as follows. For pedagogical purpose and to emphasize the main ideas that guided our computations, we first consider in section \ref{Drude} the case of the (single)  dissipative Drude model that corresponds to the particular case of \eqref{ODE_Lorentz} when $N_e = N_m = 1$ and $\omega_{e,1} = \omega_{e,m} = 0$ (this corresponds to the toy problem considered in the section 4.4 of \cite{cas-jol-kach-17}). In section \ref{Lorentz_model}, we shall extend the technical developments of the previous section to the generalized Lorentz models \eqref{ODE_Lorentz}, emphasizing the changes to be done in order to treat this more general model. Our stability results are compared with the ones of \cite{Nicaise2020}. In section \ref{sec-Extensions-result}, we  present how to apply our method  to bounded domains (section \ref{bounded_domains}) and  extend  our results to generalized  Drude-Lorentz models  (section \ref{Drude_Lorentz}) . Finally, the appendix section \ref{sec.appendix} gives the proofs of technical  results used through the paper.
\section{The case of the Drude model} \label{Drude}
In this section we are interested in studying the behaviour for long times of the solutions of the electric and magnetic fields, respectively, $\mathbf{E}$ and $\mathbf{H},$ of the Drude model whose dissipative formulation is obtained by introducing the time-derivatives of the polarization term $\mathbf{P}$ and the magnetization term $\mathbf{M}$ (where for the Drude model $\textbf{P}_\mathrm{tot}=\varepsilon_0\,\Omega_e^2 \textbf{P}$ and $\textbf{M}_\mathrm{tot}=\mu_0\,\Omega_m^2 \textbf{M}$). The unknowns of the problem are
\begin{align*}
\left\{
\begin{array}{ll}
     \textbf{E}(\bx,t):\R^3\times\R^+\longrightarrow\R^3, &\textbf{H}(\bx,t):\R^3\times\R^+\longrightarrow\R^3,\\ [6pt]
{\partial_t \mathbf{P}(\bx,t)}:\R^3\times\R^+\longrightarrow\R^3, & {\partial_t\mathbf{M}(\bx,t)}:\R^3\times\R^+\longrightarrow\R^3,
\end{array}
\right.
\end{align*}
and satisfy the governing equations 
\begin{subequations}\label{planteamiento} 
\begin{empheq}[left=\empheqlbrace]{align}
    &\varepsilon_0\,\partial_t\,\textbf{E}-\nabla\times\textbf{H}+\varepsilon_0\,\Omega_e^2\,\partial_t\,\textbf{P}=0, &(\bx,t)\in\R^3\times\R^+\label{E}\\
&\mu_0\,\partial_t\,\textbf{H}+\nabla\times\textbf{E}+\mu_0\,\Omega_m^2\,\partial_t\,\textbf{M}=0, & (\bx,t)\in\R^3\times\R^+\label{H}\\
&\partial_t^2\,\mathbf{P}+\alpha_e\, \partial_t \mathbf{P}=\mathbf{E}, & (\bx,t)\in\R^3\times\R^+\label{P}\\
&\partial_t^2\,\mathbf{M}+\alpha_m\,\partial_t \mathbf{M}=\mathbf{H}, & (\bx,t)\in\R^3\times\R^+\label{M}
\end{empheq}
\end{subequations}
completed by initial conditions
\begin{equation} \label{CI-D}
\left\{\begin{array}{l}
\mathbf{E}(\cdot, 0) =  \mathbf{E}_0, \quad \mathbf{H}(\cdot, 0) =  \mathbf{H}_0, \\ [8pt]
 \partial_t \mathbf{P}(\cdot, 0) = \partial_t \mathbf{M}(\cdot, 0) = 0. \end{array} \right.
\end{equation}
\noindent In \eqref{planteamiento}, the coefficients $(\varepsilon_0,\mu_0,\Omega_e, \Omega_m)$ are strictly positive. The coefficients $(\alpha_e, \alpha_m)$ are strictly positive damping coefficients. 
\\ [12pt] 
Setting $\mathcal{H}=L^2(\R^3)^3 \times L^2(\R^3)^3 \times L^2(\R^3)^3\times L^2(\R^3)^3$, the proposition \eqref{prop.wellposdness} in the appendix  (see also  remark \ref{rem.Drude})  insures that for $(\bE_0, \bH_0)\in L^2(\R)^3 \times L^2(\R)^3$, the system admits a unique mild solution $\bU=(\bE, \bH, \partial_t\bP, \partial_t \bM)$ in $C^{0}(\R^+, \mathcal{H})$ which is a strong solution in $C^{1}(\R^+, \mathcal{H})$ as soon as $(\bE_0, \bH_0)$ belongs to $H^1(\R^3)^3 \times H^1(\R^3)^3$. The goal of what follows is to analyze their influence on the long time behaviour (decay) of the solution. 
More precisely, our goal is in particular to obtain decay rates for the standard electromagnetic energy defined as following \begin{equation} \label{def_energyEM}
\mathcal{E}(t)\equiv \mathcal{E}(\mathbf{E}, \mathbf{H}, t):=\frac{1}{2}\left(\varepsilon_0\,\sn{\mathbf{E(\cdot,t)}}{L^2(\R^3)}^2+\mu_0\,\sn{\mathbf{H(\cdot,t)}}{L^2(\R^3)}^2\right).\end{equation}
This will be done through the following (augmented) energy
\begin{equation}\label{total_energy}
\mathcal{L}(t) \equiv \mathcal{L}(\mathbf{E}, \mathbf{H}, \mathbf{P}, \mathbf{M},t)  :=\mathcal{E}(\mathbf{E}, \mathbf{H}, t)+\frac{1}{2}\left(\varepsilon_0\,\Omega_e^2\,\sn{\mathbf{P}(\cdot,t)}{L^2(\R^3)}^2+\mu_0\,\Omega_m^2\,\sn{\mathbf{M}(\cdot,t)}{L^2(\R^3)}^2\right),
\end{equation}
that is a decreasing function of time, according to the energy identity 
\begin{equation}\label{decay_total_energy}
{\frac{d}{dt} \; \mathcal{L}(t) }+ \alpha_e \, \varepsilon_0\,\Omega_e^2\, \int_{\R^3} |\mathbf{P}(\bx,t)|^2  \, \rmd\bx + \mu_0\,   {\alpha_m}\,  \Omega_m^2\, \int_{\R^3} |\mathbf{M}(\bx,t)|^2  \, \rmd\bx = 0,
\end{equation}
that is easily demonstrated by standard arguments (see \cite{cas-jol-kach-17}, {section 4.4.1} and remark \ref{rem_energy}).
\\ [12pt]
For this purpose we shall use a Lyapunov function approach, based on the use the 3D spatial Fourier transform $\mathcal{F}$ defined  by:
$$
\mathbb{G}(\bk)=\mathcal{F}(\bG)(\bk)=\frac{1}{(2\pi)^{\frac{3}{2}} } \int_{\R^3}\bG(\bx)\,  \mathrm{e}^{-\rmi \bk \cdot \bx} \,\rmd \bx  \quad \forall \; \bG \in L^1(\R^3)^3 \cap  L^2(\R^3)^3,
$$
where we denote $\bk \in \R^3$ the dual variable of $\bx$ (or wave vector). We recall that $\mathcal{F}$ extends by density as a unitary transformation from $ L^2(\R_{\bx}^3)^3$ to $L^2(\R_{\bk}^3)^3$. We set $\mathbb{E}(\bk,t):=\mathcal{F}\{\mathbf{E}(\cdot,t)\}(\bk)$ and analogously $\mathbb{H}, \mathbb{P}$ and $\mathbb{M}$ the Fourier transforms of $\mathbf{H}, \mathbf{P}$ and $\mathbf{M},$ respectively. Accordingly we denote $\mathbb{E}_0$ the Fourier transform of $\bE_0$ and  $\mathbb{H}_0$ the Fourier transform of $\bH_0$. Then  $(\mathbb{E}, \mathbb{H}, \mathbb{P}, \mathbb{M})$ satisfy
\begin{subequations}\label{planteamiento Fourier} 
\begin{empheq}[left=\empheqlbrace]{align}
&\varepsilon_0\,\partial_t\,\mathbb{E}-\rmi\,(\bk\times\mathbb{H})+\varepsilon_0\,\Omega_e^2\,\partial_t\, \mathbb{P}=0,\label{E fourier} \\
&\mu_0\,\partial_t\,\mathbb{H}+\rmi\,(\bk\times\mathbb{E})+\mu_0\,\Omega_m^2\,\partial_t\,\mathbb{M}=0,\label{H fourier} \\
&\partial_t^2\,\mathbb{P}+\alpha_e\, \partial_t \mathbb{P}=\mathbb{E}, \label{P fourier}\\
&\partial_t^2\,\mathbb{M}+\alpha_m\,\partial_t \mathbb{M}=\mathbb{H}. \label{M fourier}
\end{empheq}
\end{subequations}
According to \eqref{CI-D}, the above system in complemented by 
\begin{equation} \label{CI-D-Fourier}
{\mathbb{E}( \bk,0) =  \mathbb{E}_0(\bk), \quad \mathbb{H}( \bk,0) =  \mathbb{H}_0(\bk), \quad
\partial_t \mathbb{P}(\bk, 0)= \partial_t \mathbb{M}(\bk, 0) = 0. }
\end{equation}
Let us note that by multiplying \eqref{E fourier} by $\overline{\mathbb{E}},$ using \eqref{P fourier} and taking the real part, we obtain that
\begin{equation}\label{1}
\frac{1}{2}\frac{d}{dt}\left(\varepsilon_0\,|\mathbb{E}(\bk,t)|^2+\varepsilon_0\,\Omega_e^2\,|\partial_t\,\mathbb{P}(\bk,t)|^2\right)+\alpha_e\,\varepsilon_0\,\Omega_e^2\,|\partial_t\,\mathbb{P}(\bk,t)|^2-\operatorname{Re} (\rmi\,(\bk\times \mathbb{H}(\bk,t))\cdot \overline{\mathbb{E}(\bk,t)})=0.
\end{equation}
Analogously we have from \eqref{H fourier} and \eqref{M fourier} that
\begin{equation}\label{2}
\frac{1}{2}\frac{d}{dt}\left(\mu_0\,|\mathbb{H}(\bk,t)|^2+\mu_0\,\Omega_m^2\,|\partial_t\,\mathbb{M}(\bk,t)|^2\right)+\alpha_m\,\mu_0\,\Omega_m^2\,|\partial_t\,\mathbb{M}(\bk,t)|^2+\operatorname{Re} (\rmi\,(\bk\times \mathbb{E}(\bk,t))\cdot \overline{\mathbb{H}(\bk,t)})=0.
\end{equation}
Finally by adding up \eqref{1} and \eqref{2} and using 
$\ds
\operatorname{Re} (\rmi\,(\bk\times \mathbb{E}(\bk,t))\cdot \overline{\mathbb{H}(\bk,t)})-\operatorname{Re} (\rmi\,(\bk\times\mathbb{H}(\bk,t))\cdot \overline{\mathbb{E}(\bk,t)})=0,
$
we show that
\begin{equation}\label{bonito}
\frac{d}{dt}\,\mathcal{L}_\bk +\mathcal{D}_{\boldsymbol{\alpha}, \bk}=0,
\end{equation}
where we have set 
\begin{equation} \label{defenergy_densities}
\left\{\begin{array}{ll}
\mathcal{L}_\bk(t) = \mathcal{E}_\bk(t) + \mathcal{E}_{\boldsymbol{\Omega}, \bk}(t), & \quad \mbox{the Lyapunov density,}\\ [12pt]
\mathcal{E}_\bk(t) := \frac{1}{2}  \big(\varepsilon_0\,|\mathbb{E}(\bk,t)|^2+\mu_0\,|\mathbb{H}(\bk,t)|^2\big), & \quad  \mbox{the energy density,} \\ [12pt]
	\mathcal{E}_{\boldsymbol{\Omega}, \bk}(t) := \frac{1}{2}  \big(\varepsilon_0\,\Omega_e^2\,|\partial_t \mathbb{P}(\bk,t)|^2+ \mu_0\,\Omega_m^2\,|\partial_t \mathbb{M}(\bk,t)|^2 \big), & \quad \mbox{the additional energy density,}
\end{array} \right.
\end{equation}
and the decay density (the index $\boldsymbol{\alpha}$ is here to emphasize the fact that this is the term in \eqref{bonito} which involes the damping coefficients $\alpha_e$ and $\alpha_m$)
\begin{equation} \label{defdecay_density}
\mathcal{D}_{\boldsymbol{\alpha}, \bk}(t) := \alpha_e\,\varepsilon_0\,\Omega_e^2\,|\partial_t  \mathbb{P}(\bk,t)|^2+\alpha_m\,\mu_0\,\Omega_m^2\,|\partial_t \mathbb{M}(\bk,t)|^2.
\end{equation}
We employ the term "density" to refer to the fact that one works at fixed $\bk$: the electromagnetic energy $\mathcal{E}$, for instance, is obtained, via Plancherel's theorem, by integration aver $\bk$ of the energy density $\mathcal{E}_\bk$
\begin{equation} \label{energy_Fourier}
\mathcal{E}(t)  = \int \mathcal{E}_\bk(t)  \; d \bk.
\end{equation}
\begin{rem} \label{rem_energy}

The reader will note that \eqref{decay_total_energy} can be recovered by integrating the above identity over $\bk \in \R^3$ and applying Plancherel's theorem. 
\end{rem}
\noindent We can not exploit only \eqref{bonito} for studying the long time behaviour of $\mathcal{L}_\bk(t)$ because we can not estimate $\mathcal{L}_\bk(t)$ with the help of $\mathcal{D}_{\boldsymbol{\alpha}, \bk}(t) $,  which does not involve $\mathbb{E}$ and $\mathbb{H}$. On the other hand, we see on \eqref{planteamiento Fourier} that, in order to control $\mathbb{E}$ and $\mathbb{H}$, we need the second order derivatives $\partial_t^2 \mathbb{P}, \partial_t^2 \mathbb{M}$. This suggests to look at the system obtained after time differentiation of \eqref{planteamiento Fourier}.  
Then in the same way that we obtained \eqref{bonito}, we have (with obvious notation)
\begin{equation}\label{bonito1}
\frac{d}{dt}\,\mathcal{L}_\bk^1+\mathcal{D}_{\boldsymbol{\alpha}, \bk}^1=0,
\end{equation}
where we have introduced the first order densities (where "first order" refers to that fact that first order derivatives of the electromagnetic field are involved)
\begin{equation} \label{defenergy_densities}
\left\{\begin{array}{ll}
\mathcal{L}^1 _\bk(t)= \mathcal{E}^1 _\bk(t) + \mathcal{E}_{\boldsymbol{\Omega}, \bk}^1(t), & \quad \mbox{first order Lyapunov density,}\\ [12pt]
\mathcal{E}_\bk^1(t) := \frac{1}{2}  \big(\varepsilon_0\,|\partial_t \mathbb{E}(\bk,t)|^2+\mu_0\,|\partial_t\mathbb{H}(\bk,t)|^2\big) & \quad  \mbox{first order energy density,} \\ [12pt]
\mathcal{E}_{\boldsymbol{\Omega}, \bk}^1(t) := \frac{1}{2}  \big(\varepsilon_0\,\Omega_e^2\,|\partial_t^2 \mathbb{P}(\bk,t)|^2+ \mu_0\,\Omega_m^2\,|\partial_t^2 \mathbb{M}(\bk,t)|^2 \big), & \quad \mbox{first order add. energy density,}
\end{array} \right.
\end{equation}
and the first order decay density
\begin{equation} \label{defdecay_density}
\mathcal{D}^1 _{\boldsymbol{\alpha}, \bk}(t) := \alpha_e\,\varepsilon_0\,\Omega_e^2\,|\partial_t ^2 \mathbb{P}(\bk,t)|^2+\alpha_m\,\mu_0\,\Omega_m^2\,|\partial_t^2\mathbb{M}(\bk,t)|^2,
\end{equation}
which precisely involves $\partial_t^2 \mathbb{P}, \partial_t^2 \mathbb{M}$. The next step is to combine \eqref{bonito} and \eqref{bonito1}. In what follows, we shall use some standard notation, to begin with
\begin{equation} \label{notation} \langle \bk\rangle = \big(1 + |\bk|^2 \big)^{\frac{1}{2}}
\end{equation} 
and to compare two positive functions $f(\bk,t)$ and $g(\bk,t)$, the notation
\begin{equation} \label{notation2}
f(\bk,t) \lsim g(\bk,t) \quad \Longleftrightarrow \quad \exists\;  C > 0 \mbox{ such that} \quad \forall \; \bk\in \R^3, \; \forall \; t\geq 0, \quad f(\bk,t) \leq C \;  g(\bk,t) .
\end{equation} 
Performing the linear combination \eqref{bonito} $+ \langle \bk\rangle^{-2} $ \eqref{bonito1} 
we obtain (cf. remark \ref{rem_norm})
\begin{equation}\label{mas bonito}
\frac{d}{dt}\,\mathcal{L}_\bk^{(1)}+\mathcal{D}_{\boldsymbol{\alpha}, \bk}^{(1)}=0,
\end{equation}
where we have introduced the first order cumulated densities for the Lyapunov function, the energy and the additional energy, (note the difference of notation between $\mathcal{L}_\bk^{(1)}$ and $\mathcal{L}_\bk^1$, etc ...)
\begin{equation} \label{defcumenergy_densities}
\mathcal{L}_\bk^{(1)} := \mathcal{L} _\bk + \langle \bk\rangle^{-2} \, \mathcal{L}^1 _\bk, \quad 
\mathcal{E}_\bk^{(1)} := \mathcal{E} _\bk + \langle \bk\rangle^{-2} \, \mathcal{E}^1 _\bk,  \quad 
\mathcal{E}_{\boldsymbol{\Omega}, \bk}^{(1)} := \mathcal{E}_{\boldsymbol{\Omega}, \bk}^1  + \langle \bk\rangle^{-2}\, \mathcal{E}_{\boldsymbol{\Omega}, \bk}^1,
\end{equation}
and the first order cumulated decay density
\begin{equation} \label{defcumdecay_density}
\mathcal{D}^{(1)}_{\boldsymbol{\alpha}, \bk}(t) := \mathcal{D}_{\boldsymbol{\alpha}, \bk}(t) + \langle \bk\rangle^{-2}\, \mathcal{D}^1_{\boldsymbol{\alpha}, \bk}(t).
\end{equation}
\begin{rem} \label{rem_norm} The weight $\langle \bk\rangle^{-2} $ is here to compensate the  time derivation in the expressions of $\mathcal{L}^1 _\bk$ and $\mathcal{D}^1_{\boldsymbol{\alpha}, \bk}$ so that the quantities summed in the definitions \eqref{defcumenergy_densities} and \eqref{defcumdecay_density}  
 have the same ``homogeneity''. 
\end{rem}
\noindent The key point is that \eqref{mas bonito} can be exploited thanks to the following lemma.
\begin{lemma}\label{lem-gronwall} 
Assume that $\alpha_e > 0$ and $\alpha_m > 0$. Then the following estimate holds
\begin{equation}\label{estimacion chida}
\mathcal{L}_\bk^{(1)}(t) \lsim \langle \bk\rangle^2\;\mathcal{D}^{(1)}_{\boldsymbol{\alpha}, \bk}(t).
\end{equation}
\end{lemma}
\begin{proof}
In what follows, for conciseness, we omit to mention the (implicit) dependence of various quantities with respect to $t$ and/or $\bk$. We recall that $\mathcal{L} _\bk^{(1)} = \mathcal{E} _\bk^{(1)} + \mathcal{E}_{\boldsymbol{\Omega}, \bk}^{(1)}$.\\ [12pt]
Since both $\alpha_e$ and $\alpha_m$ are strictly positive, we immediately observe that we can control the cumulated energy additional $\mathcal{E}_{\boldsymbol{\Omega}, \bk}^{(1)}$ with  $\mathcal{D}_{\boldsymbol{\alpha}, \bk}^{(1)}$ :
\begin{equation} \label{estiEkad1cum}
\mathcal{E}_{\boldsymbol{\Omega}, \bk}\lsim \mathcal{D}_{\boldsymbol{\alpha}, \bk}\;  \mbox{ and } \; \mathcal{E}_{\boldsymbol{\Omega}, \bk}^1 \lsim \mathcal{D}_{\boldsymbol{\alpha}, \bk}^1 \; \mbox{ which implies, by (\ref{defcumenergy_densities}, \ref{defcumdecay_density})} , \quad \mathcal{E}_{\boldsymbol{\Omega}, \bk}^{(1)} \lsim \mathcal{D}_{\boldsymbol{\alpha}, \bk}^{(1)}.
\end{equation}
For estimating the energy density $\mathcal{E} _\bk$, it is natural to use the constitutive equations \eqref{P fourier} and \eqref{M fourier} to deduce, again because $\alpha_e$ and $\alpha_m$ are strictly positive,
\begin{equation} \label{estiEk0} 
|\mathbb{E}|^2\lsim |\partial_ t\mathbb{P}|^2+|\partial_t^2\,\mathbb{P}|^2 \lsim \mathcal{D}_{\boldsymbol{\alpha}, \bk}  + \mathcal{D}_{\boldsymbol{\alpha}, \bk}^1 ,\quad 
|\mathbb{H}|^2\lsim |\partial_ t\mathbb{M}|^2+|\partial_t^2\,\mathbb{M}|^2 \lsim \mathcal{D}_{\boldsymbol{\alpha}, \bk}  + \mathcal{D}_{\boldsymbol{\alpha}, \bk}^1.
\end{equation} 
Thus, by definition of $\mathcal{E} _\bk$, 
\begin{equation} \label{estiEk} 
 \mathcal{E} _\bk \lsim \mathcal{D}_{\boldsymbol{\alpha}, \bk}  + \mathcal{D}_{\boldsymbol{\alpha}, \bk}^1.
\end{equation} 
For estimating the first order  energy density $\mathcal{E} _\bk^1$, we need to bound $\partial_t \mathbb{E}$ and $\partial_t \mathbb{H}$ for which it is natural to use Maxwell equations \eqref{E fourier} and \eqref{H fourier}. This is where we need to introduce $\bk$-dependent coefficients in our estimates. Indeed, from \eqref{E fourier} and \eqref{H fourier},
\begin{equation} \label{estiEk10}
\begin{array}{ll}
|\partial_t\,\mathbb{E}|^2 \lsim |\bk\times\mathbb{H}|^2+|\partial_ t\mathbb{P}|^2
 \lsim |\bk|^2\,|\mathbb{H}|^2+|\partial_t \mathbb{P}|^2, \\ [12pt]
 |\partial_t\,\mathbb{H}|^2 \lsim |\bk\times\mathbb{E}|^2+|\partial_ t\mathbb{M}|^2
 \lsim |\bk|^2\,|\mathbb{E}|^2+|\partial_t \mathbb{M}|^2.
\end{array}
\end{equation} 
Adding the two inequalities yields, by definition of  $\mathcal{E} _\bk^1$ \eqref{defenergy_densities} and $\mathcal{D}_{\boldsymbol{\alpha}, \bk}$ \eqref{defdecay_density} (we use again $\alpha_e ,\alpha_m > 0$)
\begin{equation} \label{estiEk1}
\mathcal{E} _\bk^1 \lsim |\bk|^2\,\mathcal{E} _\bk +\mathcal{D}_{\boldsymbol{\alpha}, \bk} \lsim  \langle \bk\rangle^2\, \mathcal{D}_{\boldsymbol{\alpha}, \bk}  + |\bk|^2 \; \mathcal{D}_{\boldsymbol{\alpha}, \bk}^1 
\end{equation} 
where, for the second inequality, we have used \eqref{estiEk}.
Finally, performing \eqref{estiEk} $ + \; \langle \bk\rangle^{-2} $ \eqref{estiEk1}, gives
\begin{equation} \label{estiEk1cum}
\mathcal{E} _\bk^{(1)} \lsim  \mathcal{D}_{\boldsymbol{\alpha}, \bk}  + \mathcal{D}_{\boldsymbol{\alpha}, \bk}^1 \leq \langle \bk\rangle^2\, \mathcal{D}_{\boldsymbol{\alpha}, \bk}^{(1)} \, \mbox{ (since $1 \leq \langle \bk\rangle^2$)}  .
\end{equation} 
Finally, \eqref{estimacion chida} results from $\mathcal{L} _\bk^{(1)} = \mathcal{E} _\bk^{(1)} + \mathcal{E}_{\boldsymbol{\Omega}, \bk}^{(1)}$, \eqref{estiEkad1cum} and \eqref{estiEk1cum}.
\end{proof}
\noindent Lemma \eqref{estimacion chida} means the existence of a constant $\sigma=\sigma(\varepsilon_0, \mu_0, \alpha_e, \alpha_m, \Omega_e, \Omega_m) > 0$ such that 
$$\mathcal{D}_{\boldsymbol{\alpha}, \bk}^{(1)}  \geq \sigma \; \langle\bk\rangle^{-2} \; \mathcal{L} _\bk^{(1)}.$$
Combined with \eqref{bonito1}, this gives the differential inequality
$$
\frac{d}{dt}\, \mathcal{L} _\bk^{(1)} + \sigma \; \langle\bk\rangle^{-2}\, \mathcal{L} _\bk^{(1)} \leq 0
$$
which gives {by integration in time:}
\begin{equation}\label{eq.gronwall}
 \mathcal{L} _\bk^{(1)}(t)\leq \mathcal{L} _\bk^{(1)}(0) \; e^{- \frac{\sigma \, t}{\langle\bk\rangle^{2}}}.
\end{equation}
Next, we show how to control $\mathcal{L} _\bk^{(1)}(0)$ in terms on the initial data as
\begin{equation}\label{estimLinitial}
\mathcal{L} _\bk^{(1)}(0)\lsim |\mathbb{E}(\bk,0)|^2+|\mathbb{H}(\bk,0)|^2.
\end{equation}
Indeed, by definition and initial conditions \eqref{CI-D-Fourier}, 
\begin{equation}\label{L0}
\mathcal{L} _\bk(0)\lsim |\mathbb{E}(\bk,0)|^2+|\mathbb{H}(\bk,0)|^2.
\end{equation}
On the other hand, from \eqref{defenergy_densities}, one has
$$
 \mathcal{L} _\bk^1(0)\leq |\partial_t \mathbb{E}(\bk,0)|^2+|\partial_t \mathbb{H}(\bk,0)|^2 + |\partial_t^2 \mathbb{P}(\bk,0)|^2+|\partial_t^2 \mathbb{M}(\bk,0)|^2 .
$$
 By (\ref{E fourier}, \ref{H fourier}) at $t=0$, 
$
|\partial_t \mathbb{E}(\bk,0)|^2+|\partial_t \mathbb{H}(\bk,0)|^2 \lsim |\bk|^2 \; \big(|\mathbb{E}(\bk,0)|^2+|\mathbb{H}(\bk,0)|^2\big)
$ since $\partial_t \mathbb{P}(\bk,0)$ and $\partial_t \mathbb{M}(\bk,0)$ vanish,
while (\ref{P fourier},\ref{M fourier}) give
$
 |\partial_t^2 \mathbb{P}(\bk,0)|^2+|\partial_t^2 \mathbb{M}(\bk,0)|^2 \lsim |\partial_t \mathbb{E}(\bk,0)|^2+|\partial_t \mathbb{H}(\bk,0)|^2.
$ Thus
\begin{equation}\label{L10}
\mathcal{L} _\bk^1(0) \lsim |\bk|^2 \; \big(|\mathbb{E}(\bk,0)|^2+|\mathbb{H}(\bk,0)|^2\big).
\end{equation}
Finally \eqref{estimLinitial} results from \eqref{L0} $+ \langle\bk\rangle^{-2}$ \eqref{L10}. At last, substituting \eqref{estimLinitial} into \eqref{eq.gronwall}, as $\mathcal{L} _\bk \leq \mathcal{L} _\bk^{(1)}$
\begin{equation} \label{estiLk}
\mathcal{L} _\bk(t)\leq \big(|\mathbb{E}(\bk,0)|^2+|\mathbb{H}(\bk,0)|^2\big) \;  e^{- \frac{\sigma \, t}{\langle\bk\rangle^{2}}} \; .
\end{equation}
This inequality says that the Lyapunov density decays exponentially in time for each $|\bk|$. It also says the the decay rate depends on $|\bk|$ and tends to $0$ when $|\bk|$ tends  to $+ \infty$. This is the reason why, when coming back to the augmented energy $\mathcal{L}(t)$ (cf. \eqref{total_energy}) we shall obtain only polynomial decay.
 This is summarized in the main theorem of this section:
\begin{theorem} \label{thm_Drude}
For any 
$(\mathbf{E}_0, \mathbf{H}_0) \in L^2(\R^3)^3 \times L^2(\R^3)^3$, the total energy tends to $0$ when $t$ tends to $+\infty :$
\begin{equation} \label{convergence}
\lim_{t \rightarrow + \infty} \mathcal{L}(t) = 0.
\end{equation}
If $
(\mathbf{E}_0, \mathbf{H}_0) \in H^{m}(\R^3)^3 \times  H^{m}(\R^3)^3 \quad \mbox{for some  integer } m > 0$, 
one has a polynomial decay rate
\begin{equation} \label{polynomial_decay}
\mathcal{L}(t)\lsim \Big( \sn{\mathbf{E}_0}{H^{m}(\R^3)}^2 + \sn{\mathbf{H}_0}{H^{m}(\R^3)}^2 \Big)\;t^{-m}.
\end{equation}
\end{theorem}
\begin{proof}
From the respective definitions of 	$\mathcal{L}$ and $\mathcal{L}_\bk$, by Plancherel's identity and \eqref{estiLk}, we have
	\begin{equation} \label{energy_fourier-drude}
	\mathcal{L}(t)=\int_{\R^3} \mathcal{L}_\bk(t) \; \rmd \bk \lsim \int_{\R^3} \big( \, |\mathbb{E}_0(\bk)|^2 + |\mathbb{H}_0(\bk)|^2 \, \big) \;  e^{- \frac{\sigma \, t}{\langle\bk\rangle^{2}}} \; \rmd \bk.
	\end{equation}
From Lebesgue's dominated convergence theorem, we first conclude that \eqref{convergence} holds for any initial data $(\mathbf{E}_0, \mathbf{H}_0) \in L^2(\R^3)^3 \times L^2(\R^3)^3$.\\[12pt]
Next, in order to exploit the Sobolev regularity of the initial for obtaining \eqref{polynomial_decay}, we rewrite \eqref{energy_fourier-drude} as follows (we simply make appear artificially the factor $\langle\bk\rangle^{2m}/t^m$)
$$
\begin{array}{lll}
\mathcal{L}(t) & \lsim & \ds t^{-m} \int_{\R^3} \langle\bk\rangle^{2m} \Big( |\mathbb{E}_0(\bk)|^2 + |\mathbb{H}_0(\bk)|^2\Big) \;  \big({t}/{\langle\bk\rangle^{2}}\big)^m \; e^{- \frac{\sigma t}{\langle\bk\rangle^{2}}} \; \rmd \bk \\ [18pt] 
& \lsim & \ds   t^{-m} \int_{\R^3} \langle\bk\rangle^{2m} \Big( |\mathbb{E}_0(\bk)|^2 + |\mathbb{H}_0(\bk)|^2\Big) \;  F\big(t/\langle\bk\rangle^{2}\big) \; \rmd \bk. \end{array}
$$
where we have set $F_m(r) := r^{m} \, e^{- \sigma r}, r \geq 0$ which is clearly bounded on $\R^+$. \\ [12pt]
Setting  $C_m:= \, \displaystyle\sup_{r\geq 0} F_m(r)=(m/(\sigma \,\mathrm{e}))^{m}$, by Fourier characterization of Sobolev norms, we get \\\centerline{$\ds
\mathcal{L}(t) \lsim C_m\,  t^{-m} \; \Big( \sn{\mathbf{E}_0}{H^{m+1}(\R^3)}^2 + \sn{\mathbf{H}_0}{H^{m+1}(\R^3)}^2 \Big).
$}
\end{proof}
\section{The case of the generalized Lorentz model} \label{Lorentz_model}
In this section, our goal is to extend the results in the latest section to the case of the (generalized) Lorentz model. The evolution (Cauchy) problem reads as follows
\begin{align*}
\mbox{Find} \quad \quad \left\{
\begin{array}{ll}
     \textbf{E}(\bx,t):\R^3\times\R^+\longrightarrow\R^3 &\textbf{H}(\bx,t):\R^3\times\R^+\longrightarrow\R^3\\ [6pt]
  \mathbf{P}_j(\bx,t):\R^3\times\R^+\longrightarrow\R^3, 1 \leq j \leq N_e, \quad  &\mathbf{M}_\ell(\bx,t):\R^3\times\R^+\longrightarrow\R^3,  \quad 1 \leq \ell \leq N_m,
\end{array}
\right.
\end{align*}
such that (for all $1 \leq j \leq N_e, 1 \leq \ell \leq N_m$)
\begin{subequations}\label{planteamiento Lorentz} 
\begin{empheq}[left=\empheqlbrace]{align}
    &\varepsilon_0\,\partial_t\,\textbf{E}-\nabla\times\textbf{H}+\varepsilon_0\,\sum_{j=1}^{N_e}\,\Omega_{e,j}^2\,\partial_t\,\mathbf{P}_j=0, &(\bx,t)\in\R^3\times\R^+,\label{E Lorentz}\\
&\mu_0\,\partial_t\,\textbf{H}+\nabla\times\textbf{E}+\mu_0\,\sum_{\ell=1}^{N_m}\,\Omega_{m,\ell}^2\,\partial_t\,\mathbf{M}_\ell=0, & (\bx,t)\in\R^3\times\R^+,\label{H Lorentz}\\[6pt]
&\partial_t^2\,\mathbf{P}_j+\alpha_{e,j}\,\partial_t\,\mathbf{P}_j+\omega_{e,j}^2\,\mathbf{P}_j=\mathbf{E}, & (\bx,t)\in\R^3\times\R^+,\label{P Lorentz}\\[12pt]
&\partial_t^2\,\mathbf{M}_\ell+\alpha_{m,\ell}\,\partial_t\,\mathbf{M}_j+\omega_{m,\ell}^2\,\mathbf{M}_j=\mathbf{H}, & (\bx,t)\in\R^3\times\R^+,\label{M Lorentz}
\end{empheq}
\end{subequations}
completed by the following divergence free initial conditions
\begin{equation} \label{CI}
\left\{\begin{array}{l}
\mathbf{E}(\cdot, 0) =  \mathbf{E}_0, \quad \mathbf{H}(\cdot, 0) =  \mathbf{H}_0 \quad  \mbox{ with } \quad \nabla \cdot \mathbf{E}_0=\nabla \cdot \mathbf{H}_0=0, \\ [8pt]
\mathbf{P}(\cdot, 0) =\mathbf{M}(\cdot, 0)  =\partial_t \mathbf{P}(\cdot, 0) = \partial_t \mathbf{M}(\cdot, 0) = 0. \end{array} \right.
\end{equation}
where  for simplifying some formulas, we  treat the $\mathbf{P}_j$ and the $\mathbf{M}_\ell$ collectively setting 
\begin{equation} \label{notPM} {\mathbf{P}}=(\mathbf{P}_j)_{j=1}^{N_e} \quad \mbox{and} \quad {\mathbf{M}}=(\mathbf{M}_\ell)_{\ell=1}^{N_m}.
\end{equation}
In the above equations, the coefficients $(\Omega_{e,j},\Omega_{m,\ell}, \Omega_{e,j},\Omega_{m,\ell}) $ are supposed to satisfy (see remark \ref{rem_Drude})
\begin{equation} \label{hypomegabis}
\begin{array}{l}
0 < \Omega_{e,1} \leq \cdots \leq \Omega_{e,N_e}, \quad 0 < \Omega_{m,1} \leq \cdots \leq \Omega_{m,N_m},  \\ [12pt]
\omega_{e,j} > 0, \quad 1 \leq j \leq N_e, \quad  \omega_{m,\ell} > 0,\quad 1 \leq \ell \leq N_m.
\end{array}
\end{equation}
and one can assume without any loss of generality that the couples $(\alpha_{e,j}, \omega_{e,j})$ (resp. $(\alpha_{m,\ell}, \omega_{m,\ell})$) are all distinct the ones from the others.
\begin{rem} \label{rem_Drude} Note that one recovers the Drude model of section with
$
N_e = N_m = 1 \mbox{  if } \quad \omega_{e,1} = \omega_{m,1} = 0.
$
\end{rem}
\noindent Setting $
\mathcal{H} =L^2(\mathbb{R}^3)^3\times L^2(\mathbb{R}^3)^3 \times L^2(\mathbb{R}^3)^{3N_e}
\times L^2(\mathbb{R}^3)^{3N_e} \times L^2(\mathbb{R}^3)^{3N_m}  \times L^2(\mathbb{R}^3)^{3N_m} 
$, the proposition \ref{prop.wellposdness} insures that for  $(\bE_0, \bH_0)\in L^2(\R)^3$, the system  \eqref{planteamiento Lorentz} admits a unique mild solution $\bU=(\bE, \bH, \partial_t\bP, \partial_t \bM)$ in $C^{0}(\R^+, \mathcal{H})$, which is a strong solution in $C^{1}(\R^+, \mathcal{H})$ as soon as $(\bE_0, \bH_0)\in H^1(\R^3)^3 \times H^1(\R^3)^3$.\\

\noindent {The equations \eqref{planteamiento Lorentz} are
completed by the initial conditions \eqref{CI}.  If the initial electric and magnetic fields are divergence free, all vector fields appearing in \eqref{planteamiento Lorentz} are divergence free at any time, {see Proposition \ref{prop.div} of appendix \ref{sec-apend-D}}. 
\begin{equation}\label{free divergenceLorentz}
\nabla\cdot\mathbf{E}(\cdot, t)= \nabla\cdot\mathbf{H} (\cdot, t)= \nabla\cdot  \mathbf{P}_j (\cdot, t)= \nabla\cdot  \mathbf{M}_{\ell}(\cdot, t)=0, \quad  \forall  \; t>0, \quad \forall \; j,m.
\end{equation}}
The equivalent for Lorentz of the identity \eqref{decay_total_energy} (for Drude) is (see \cite{cas-jol-kach-17})
\begin{equation}\label{decay_total_energyL}
{\frac{d}{dt} \; \mathcal{L}(t) }+  \varepsilon_0 \sum_{j=1}^{N_e} \alpha_{j,e} \,\Omega_{j,e}^2\, \int_{\R^3} |\mathbf{P}_j(\bx,t)|^2  \, \rmd\bx + \mu_0\,  \sum_{\ell=1}^{N_m} \alpha_{m,\ell}\,  \Omega_{m,\ell}^2\, \int_{\R^3} |\mathbf{M}_\ell(\bx,t)|^2  \, \rmd\bx = 0.
\end{equation}
We assume the strong dissipation assumption \eqref{SD}, namely all the damping coefficients are positive:
$
\alpha_{e,j}, \alpha_{m,\ell} >0 \;\mbox{ for all }1\leq j\leq N_e, 1\leq \ell\leq N_m.
$ \\ [12pt] 
Next, we observe that the space Fourier transforms $(\mathbb{E}, \mathbb{H}, \mathbb{P}_j, \mathbb{M}_\ell)$ of $({\bf E}, {\bf H}, {\bf P}_j, {\bf M}_\ell)$  satisfy
\begin{subequations}\label{planteamiento Lorentz Fourier} 
\begin{empheq}[left=\empheqlbrace]{align}
    &\varepsilon_0\,\partial_t\,\mathbb{E}-i\, \bk\times\mathbb{H}+\varepsilon_0\,\sum_{j=1}^{N_e}\,\Omega_{e,j}^2\,\partial_t\,\mathbb{P}_j=0, \label{E Lorentz Fourier}\\
&\mu_0\,\partial_t\,\mathbb{H}+i\,\bk\times\mathbb{E}+\mu_0\,\sum_{\ell=1}^{N_m}\,\Omega_{m,\ell}^2\,\partial_t\,\mathbb{M}_\ell=0, \label{H Lorentz Fourier}\\[6pt]
&\partial_t^2\,\mathbb{P}_j+\alpha_{e,j}\,\partial_t\,\mathbb{P}_j+\omega_{e,j}^2\,\mathbb{P}_j=\mathbb{E}, \label{P Lorentz Fourier}\\[12pt]
&\partial_t^2\,\mathbb{M}_\ell+\alpha_{m,\ell}\,\partial_t\,\mathbb{M}_\ell+\omega_{m,\ell}^2\,\mathbb{M}_\ell=\mathbb{H} \label{M Lorentz Fourier}. 
\end{empheq}
\end{subequations}
According to \eqref{notPM}, we shall set
\begin{equation} \label{notPMF}
{\mathbb{P}}:=(\mathbb{P}_j)_{j=1}^{N_e}, \quad 
{\mathbb{M}}:=(\mathbb{M}_\ell)_{\ell=1}^{N_m},
\end{equation}
and will also use the condensed notation
\begin{equation} \label{normPMF} 
|\mathbb{P}|^2 := \sum _{j=1}^{N_e} |\mathbb{P}_j|^2, \quad \mbox{and} \quad |\mathbb{M}|^2 := \sum _{\ell=1}^{N_e} |\mathbb{P}_\ell|^2.
\end{equation}
To study the long time behaviour of the solution of \eqref{planteamiento Lorentz Fourier}, it is natural to try to use the same approach than for the Drude model in section \ref{Drude}. As a matter of fact 
proceeding as for obtaining \eqref{bonito}, we can derive from \eqref{planteamiento Lorentz Fourier}, {and the identity}
$
\operatorname{Re} (\rmi\,(\bk\times \mathbb{E}(\bk,t))\cdot \overline{\mathbb{H}(\bk,t)})-\operatorname{Re} (\rmi\,(\bk\times\mathbb{H}(\bk,t))\cdot \overline{\mathbb{E}(\bk,t)})=0,
$
{the  relation}
\begin{equation}\label{bonitoL}
\frac{d}{dt}\,\mathcal{L}_\bk +\mathcal{D}_{\boldsymbol{\alpha}, \bk}=0,
\end{equation}
where we have introduced the energy densities 
\begin{equation} \label{defenergy_densitiesL}
\left\{\begin{array}{ll}
\mathcal{L}_\bk(t) = \mathcal{E}_\bk(t) + \mathcal{E}_{\boldsymbol{\Omega}, \bk}(t), \\ [14pt]
\mathcal{E}_\bk(t) := \frac{1}{2}  \big(\varepsilon_0\,|\mathbb{E}(\bk,t)|^2+\mu_0\,|\mathbb{H}(\bk,t)|^2\big) \\ [6pt]
\ds \mathcal{E}_{\boldsymbol{\Omega}, \bk}(t) := \frac{1}{2}  \bigg(\varepsilon_0\,\sum_{j=1}^{N_e} \Omega_{e,j}^2\, |\partial_t \mathbb{P}_j(\bk,t)|^2 +  \mu_0\,\sum_{\ell=1}^{N_m}\Omega_{m,\ell}^2\, | \partial_t \mathbb{M}_\ell(\bk,t)|^2\bigg) 
\\ [6pt] \hspace*{1.3cm } + \; \displaystyle\frac{1}{2}  \bigg( \varepsilon_0\,\sum_{j=1}^{N_e} \omega_{e,j}^2\, \Omega_{e,j}^2\, |\mathbb{P}_j(\bk,t)|^2 +  \mu_0\,\sum_{\ell=1}^{N_m} \omega_{m,\ell}^2\,\Omega_{m,\ell}^2\, |\mathbb{M}_\ell(\bk,t)|^2 \bigg),
\end{array} \right.
\end{equation}
and the decay density 
\begin{equation} \label{defdecay_densityL}
\mathcal{D}_{\boldsymbol{\alpha}, \bk}(t) := \varepsilon_0\,\sum_{j=1}^{N_e} \alpha_{e,j} \, \Omega_{e,j}^2\,|\partial_t \mathbb{P}_j(\bk,t)|^2+ \mu_0\,\sum_{\ell=1}^{N_m} \alpha_{e,\ell} \, \Omega_{m,\ell}^2\,|\partial_t \mathbb{M}_{\ell}(\bk,t)|^2.
\end{equation}
The main novelty, with respect to the Drude case, is the apparition of the second line in the definition of $\mathcal{E}_{\boldsymbol{\Omega}, \bk}$, that involves the fields $\mathbb{P}$ and $\mathbb{M}$ and not
only their time derivative. \\ [12pt]
Reasoning with the time derivatives of the fields as for the Drude case, we also have
\begin{equation}\label{bonito1L}
\frac{d}{dt}\,\mathcal{L}_\bk^1+\mathcal{D}_{\boldsymbol{\alpha}, \bk}^1=0,
\end{equation}
having defined the first order energy densities
\begin{equation} \label{defenergy_densities1L}
\left\{\begin{array}{ll}
\mathcal{L}_\bk^1(t) = \mathcal{E}^1_\bk(t) + \mathcal{E}^1_{\boldsymbol{\Omega}, \bk}(t), \\ [14pt]
\mathcal{E}^1_\bk(t) := \frac{1}{2}  \big(\varepsilon_0\,|\partial_t \mathbb{E}(\bk,t)|^2+\mu_0\,| \partial_t \mathbb{H}(\bk,t)|^2\big) \\ [6pt]
\ds \mathcal{E}^1_{\boldsymbol{\Omega}, \bk}(t) := \frac{1}{2}  \bigg(\varepsilon_0\,\sum_{j=1}^{N_e} \Omega_{e,j}^2\, |\partial_t^2 \mathbb{P}_j(\bk,t)|^2 +  \mu_0\,\sum_{\ell=1}^{N_m}\Omega_{m,\ell}^2\, | \partial_t^2 \mathbb{M}_\ell(\bk,t)|^2\bigg) 
\\ [6pt] \hspace*{1.3cm } + \; \displaystyle\frac{1}{2}  \bigg( \varepsilon_0\,\sum_{j=1}^{N_e} \omega_{e,j}^2\, \Omega_{e,j}^2\, |\partial_t \mathbb{P}_j(\bk,t)|^2 +  \mu_0\,\sum_{\ell=1}^{N_m} \omega_{m,\ell}^2\,\Omega_{m,\ell}^2\, | \partial_t \mathbb{M}_\ell(\bk,t)|^2 \bigg),
\end{array} \right.
\end{equation}
and the first order decay density 
\begin{equation} \label{defdecay_density1L}
\mathcal{D}^1_{\boldsymbol{\alpha}, \bk}(t) := \varepsilon_0\,\sum_{j=1}^{N_e} \alpha_{e,j} \, \Omega_{e,j}^2\,|\partial_t^2 \mathbb{P}_j(\bk,t)|^2+ \mu_0\,\sum_{\ell=1}^{N_m} \alpha_{e,\ell} \, \Omega_{m,\ell}^2\,|\partial_t^2\mathbb{M}_{\ell}(\bk,t)|^2.
\end{equation}
However, this time, 
\eqref{bonitoL} and \eqref{bonito1L} will not be sufficient to proceed as in the Drude case because
we need to contrôl the term in the second line of the definition \eqref{defenergy_densitiesL} of $\mathcal{E}_{\boldsymbol{\Omega}, \bk}$, in other words $\mathbb{P}$ and $\mathbb{M}$. Because these fields only appear the constitutive laws \eqref{P Lorentz Fourier} and \eqref{M Lorentz Fourier}, we need to adopt a different strategy with respect to section \ref{Drude}:
\begin{itemize} 
\item [(i)] This time, the constitutive laws \eqref{P Lorentz Fourier} and \eqref{M Lorentz Fourier} are used to control $\mathbb{P}$  and $\mathbb{M}$ (and no longer $\mathbb{E}$ and $\mathbb{H}$) in function of   $\mathbb{E}$, $\partial_t \mathbb{P}, \, \partial_t^2 \mathbb{P}$,  $\mathbb{H}$, $\partial_t \mathbb{M}$ and $\partial_t^2 \mathbb{M}$.
\item [(ii)] we then need to control $\mathbb{E}$ and $\mathbb{M}$ in another manner: this will be done by using the Maxwell's equations (\ref{E Lorentz Fourier},\ref{H Lorentz Fourier}), via $\bk \times \mathbb{E}$ (resp. $\bk \times \mathbb{M}$) (this control will thus  degenerate when $|\bk|$ tends to $0$) in function of $\partial_t \mathbb{H}$ and $\partial_t \mathbb{M}$ (resp. $\partial_t \mathbb{E}$ and $\partial_t \mathbb{P}$). {This will use $|\bk\times  \mathbb{E}|=|\bk|  \, | \mathbb{E}|$  (resp. $|\bk\times  \mathbb{H}|=|\bk| \,  | \mathbb{H}|$), which is the counterpart  in  Fourier space of the free divergence property \eqref{free divergenceLorentz}.}
\item [(iii)]  Finally, to control $ \partial_t \mathbb{E}$ and $\partial_t \mathbb{H}$, the idea is to use again \eqref{P Lorentz Fourier} and \eqref{M Lorentz Fourier}, but this time after time differentiation. Doing so, we control  $ \partial_t \mathbb{E}$ and $\partial_t \mathbb{H}$ with $\partial^2_t \mathbb{P}$ and $\partial^2_t \mathbb{M}$, which do appear in the definitions of $\mathcal{D}_{\boldsymbol{\alpha}, \bk}$ and $\mathcal{D}_{\boldsymbol{\alpha}, \bk}^1$, but also the third order derivatives $\partial^3_t \mathbb{P}$ and $\partial^3_t \mathbb{M}$.
\end{itemize}
That is why, in order to make appear a damping function containing the third order derivatives,  we differentiate the equations of the problem once more in time, which leads to the identity
\begin{equation}\label{bonito2L}
\frac{d}{dt}\,\mathcal{L}_\bk^2+\mathcal{D}_{\boldsymbol{\alpha}, \bk}^2=0,
\end{equation}
having defined the second order energy densities
\begin{equation} \label{defenergy_densities2L}
\left\{\begin{array}{ll}
\mathcal{L}_\bk^2(t) = \mathcal{E}^2_\bk(t) + \mathcal{E}^2_{\boldsymbol{\Omega}, \bk}(t), \\ [14pt]
\mathcal{E}^2_\bk(t) := \frac{1}{2}  \big(\varepsilon_0\,|\partial_t^2 \mathbb{E}(\bk,t)|^2+\mu_0\,| \partial_t^2 \mathbb{H}(\bk,t)|^2\big) \\ [6pt]
\ds \mathcal{E}^2_{\boldsymbol{\Omega}, \bk}(t) := \frac{1}{2}  \bigg(\varepsilon_0\,\sum_{j=1}^{N_e} \Omega_{e,j}^2\, |\partial_t^3 \mathbb{P}_j(\bk,t)|^2 +  \mu_0\,\sum_{\ell=1}^{N_m}\Omega_{m,\ell}^2\, | \partial_t^3 \mathbb{M}_\ell(\bk,t)|^2\bigg) 
\\ [6pt] \hspace*{1.3cm } + \; \displaystyle\frac{1}{2}  \bigg( \varepsilon_0\,\sum_{j=1}^{N_e} \omega_{e,j}^2\, \Omega_{e,j}^2\, |\partial_t^2 \mathbb{P}_j(\bk,t)|^2 +  \mu_0\,\sum_{\ell=1}^{N_m} \omega_{m,\ell}^2\,\Omega_{m,\ell}^2\, | \partial_t^2 \mathbb{M}_\ell(\bk,t)|^2 \bigg),
\end{array} \right.
\end{equation}
and the second order decay density 
\begin{equation} \label{defdecay_density2L}
\mathcal{D}^2_{\boldsymbol{\alpha}, \bk}(t) := \varepsilon_0\,\sum_{j=1}^{N_e} \alpha_{e,j} \, \Omega_{e,j}^2\,|\partial_t^3 \mathbb{P}_j(\bk,t)|^2+ \mu_0\,\sum_{\ell=1}^{N_m} \alpha_{e,\ell} \, \Omega_{m,\ell}^2\,|\partial_t^3 \mathbb{M}_{\ell}(\bk,t)|^2.
\end{equation}
Finally, proceeding as for obtaining \eqref{mas bonito}, we deduce from equations \eqref{bonitoL} to \eqref{defdecay_density2L} that
\begin{equation}\label{mas bonitoL}
\frac{d}{dt}\,\mathcal{L}_\bk^{(2)}+\mathcal{D}_{\boldsymbol{\alpha}, \bk}^{(2)}=0,
\end{equation}
where we have introduced the second order cumulated energy densities 
\begin{equation} \label{defcumenergy_densitiesL}
\mathcal{L}_\bk^{(2)} := \sum_{j=0}^2\langle \bk\rangle^{-2j} \, \mathcal{L}^j _\bk \equiv \mathcal{E}_\bk^{(2)} + \mathcal{E}_{\boldsymbol{\Omega}, \bk}^{(2)}, \quad 
\mathcal{E}_\bk^{(2)} := \sum_{j=0}^2\langle \bk\rangle^{-2j} \, \mathcal{E}^j _\bk,  \quad 
\mathcal{E}_{\boldsymbol{\Omega}, \bk}^{(2)} := \sum_{j=0}^2\langle \bk\rangle^{-2j} \, \mathcal{E}_{\boldsymbol{\Omega}, \bk}^j,
\end{equation}
 and the second order cumulated decay density
\begin{equation} \label{defcumdecay_densityL}
\mathcal{D}^{(2)}_{\boldsymbol{\alpha}, \bk} :=\sum_{j=0}^2\langle \bk\rangle^{-2j} \, \mathcal{D}^j _{\boldsymbol{\alpha}, \bk}.
\end{equation}
In (\ref{defcumenergy_densitiesL}, \ref{defcumdecay_densityL}), by convention, $\mathcal{L} ^0_\bk = \mathcal{L}_\bk, $ etc ... The key point is that, according to the process (i)(ii)(iii) described above, we can bound $\mathcal{L}_\bk^{(2)}$ in terms of $\mathcal{D}^{(2)}_{\boldsymbol{\alpha}, \bk}$ : 
\begin{lemma}\label{lem-gronwallL} 
	{Assume that the strong dissipation assumption holds. Then, one has  the following estimate}
	\begin{equation}\label{estimacion chidaL}
	\mathcal{L}_\bk^{(2)}(t) \lsim  \big(\langle \bk\rangle^2+ |\bk|^{-2}\big)  \; \mathcal{D}^{(2)}_{\boldsymbol{\alpha}, \bk}(t).
	\end{equation}
\end{lemma}
 \begin{proof} Before entering the technical details, let us first give the main ideas and steps of the proof. The goal is to control $\mathcal{H}^{(2)}$ with the help of $\mathcal{D}^{(2)}_{\boldsymbol{\alpha}}$. Towards this goal, we observe that
$$
\begin{array}{lll}
\mathcal{L}_\bk^{(2)}  & \mbox{is a norm in}  &\mathbb{U}:=\big(\mathbb{E}, \partial_t \mathbb{E}, \partial_t^2 \mathbb{E}, \mathbb{H}, \partial_t \mathbb{H}, \partial_t^2 \mathbb{H},\mathbb{P}, \partial_t \mathbb{P}, \partial_t^2 \mathbb{P}, \partial_t^3 \mathbb{P}, \mathbb{M}, \partial_t \mathbb{M}, \partial_t^2 \mathbb{M}, \partial_t^3 \mathbb{M}\big), 
\\ [12pt]
\mathcal{D}^{(2)}_{\boldsymbol{\alpha}, \bk} & \mbox{is a norm in} &  \mathbb{V}:=(\partial_t \mathbb{P}, \partial_t^2 \mathbb{P}, \partial_t^3 \mathbb{P},  \partial_t \mathbb{M}, \partial_t^2 \mathbb{M}, \partial_t^3 \mathbb{M}), 
\end{array}
$$
The idea that we shall develop is that, roughly speaking,  $\mathcal{D}^{(2)}_{\boldsymbol{\alpha}}$ is also a norm with respect to in $\mathbb{U}$ along the linear manifold (in the $\mathbb{U}$-space) generated by the equations $(\ref{E Lorentz Fourier}, \ref{H Lorentz  Fourier},\ref{P Lorentz Fourier},\ref{M Lorentz Fourier})$ and their time-derivatives. More precisely, we shall be able to control 
the terms which are missing in $\mathbb{V}$ namely  $\mathbb{E},\partial_t \mathbb{E}, \partial_t^2 \mathbb{E}, \mathbb{H}, \partial_t \mathbb{H}, \partial_t^2 \mathbb{H},$ $\mathbb{P}, \mathbb{M}$ with the terms appearing in $\mathcal{D}^{(2)}_{\boldsymbol{\alpha}}$ namely $\partial_t \mathbb{P}, \partial_t^2 \mathbb{P}, \partial_t^3 \mathbb{P},  \partial_t \mathbb{M}, \partial_t^2 \mathbb{M}, \partial_t^3 \mathbb{M}$. 
This will be done in the following order
$$ 
\begin{array}{ll}
 (a) &  \left\{ \begin{array}{l} \mbox{Control $\partial_t \mathbb{E}$ with $\partial_t \mathbb{P}, \partial_t^2 \mathbb{P}, \partial_t^3 \mathbb{P}$ using $\partial_t$\eqref{P Lorentz Fourier},} \\ [6pt]
  \mbox{Control $\partial_t \mathbb{H}$ with $\partial_t \mathbb{M}, \partial_t^2 \mathbb{M},
    \partial_t^3 \mathbb{M}$ using $\partial_t$\eqref{M Lorentz Fourier},}
    \end{array} \right. 
    \\ [20pt]
(b)  &  \left\{ \begin{array}{l} \mbox{Control $\mathbb{E}$ with $\partial_t \mathbb{H}, \partial_t \mathbb{M}$ using \eqref{H Lorentz Fourier},} \\ [6pt]
\mbox{Control $\mathbb{H}$ with $\partial_t \mathbb{E}, \partial_t \mathbb{P}$ using \eqref{E Lorentz Fourier},} \end{array} \right.
\\[20pt]
(c) &  \left\{ \begin{array}{l}  \mbox{Control $\partial_t^2\mathbb{E}$ with $\partial_t \mathbb{H}, \partial_t \mathbb{P}$ using $\partial_t$\eqref{E Lorentz Fourier},}\\ [6pt]
    \mbox{Control $\partial_t^2 \mathbb{H}$ with $\partial_t \mathbb{E}, \partial_t \mathbb{M}$ using $\partial_t$\eqref{H  Lorentz Fourier},}
    \end{array} \right. 
    \\[20pt]
    (d) &  \left\{ \begin{array}{l} \mbox{ Control $\mathbb{P}$ with $\partial_t \mathbb{P}, \, \partial_t^2 \mathbb{P}$ and $\mathbb{E}$ using \eqref{P  Lorentz Fourier},} \\[6pt]
\mbox{ Control $\mathbb{M}$ with $\partial_t \mathbb{M}, \, \partial_t^2 \mathbb{M}$ and $\mathbb{H}$ using \eqref{M Lorentz Fourier}.}
\end{array} \right.
\end{array}
$$
 Of course the constants in the estimates issued from using Maxwell's equations (\ref{E Lorentz Fourier}, \ref{H Lorentz Fourier}), that is to say the ones of steps (b) and (c), will be $\bk$-dependent. It is worth while mentioning that these equations are used in a different manner in (b) and (c). In step (b), they are used to estimate lower order time  derivatives of the field $(\mathbb{E}, \mathbb{H})$ with higher order time  derivatives: as a consequence, the constants in the estimates will blow up when $\bk$ tend to $0$. At the contrary, in step (c), they are used to estimate higher order time  derivatives of the field $(\mathbb{E}, \mathbb{H})$ with lower order time  derivatives: as a consequence, the constants in the estimates will blow up when $|\bk|$ tend to $+ \infty$. \\ [12pt]
Let us now enter in the details of the proof. \\[12pt]
{\it Step 1 : Control of the energy density $\mathcal{E}_\bk^{(2)}$. } We control below each ot the terms $\langle \bk\rangle^{-2j} \, \mathcal{E}^j _\bk$ appearing in the sum \eqref{defcumenergy_densitiesL} defining $\mathcal{E}_\bk^{(2)}$.\\[12pt]
(a) Case $j=1$ : control of $\langle \bk\rangle^{-2} \, \mathcal{E}^1 _\bk$. \\ [12pt]
 Using the strong dissipation  assumption  \eqref{SD}, we obtain  from differentiating \eqref{P Lorentz Fourier} with respect to time and a summation over $j$ according to  (\ref{notPMF}, \ref{normPMF})
\begin{equation} \label{estdtE}
 |\partial_t \mathbb{E}|^2 \lsim |\partial_t\mathbb{P}|^2 + |\partial_t^2\mathbb{P}|^2 + |\partial_t^3\mathbb{P}|^2.
\end{equation}
On the other hand, by definition of each $\mathcal{D}^{j}_{\boldsymbol{\alpha}, \bk}$,  and making appear at the right hand sides the terms in the sum that defines $\mathcal{D}^{(2)}_{\boldsymbol{\alpha}, \bk}$, we have
$$
\begin{array}{llll} 
|\partial_t^3 \mathbb{P}|^2 \leq \mathcal{D}^2_{\boldsymbol{\alpha}, \bk} & \; \Longrightarrow \; & \langle \bk\rangle^{-2} \, |\partial_t^3 \mathbb{P}|^2  \lsim \langle \bk\rangle^{-2} \; \mathcal{D}^2_{\boldsymbol{\alpha}, \bk}  = \langle \bk\rangle^{2 } \, (\langle \bk\rangle^{-4}  \; \mathcal{D}^2_{\boldsymbol{\alpha}, \bk}) &\\ [8pt]
|\partial_t^2 \mathbb{P}|^2 \leq \mathcal{D}^1_{\boldsymbol{\alpha}, \bk} &\; \Longrightarrow \; & \langle \bk\rangle^{-2} \, |\partial_t^2 \mathbb{P}|^2  \lsim \langle \bk\rangle^{-2} \; \mathcal{D}^1_{\boldsymbol{\alpha}, \bk}  \lsim \langle \bk\rangle^{2 } \, \big(\langle \bk\rangle^{-2} \; \mathcal{D}^1_{\boldsymbol{\alpha}, \bk} \big)  &\; (\langle \bk\rangle \geq 1)\\ [8pt]
|\partial_t \mathbb{P}|^2 \leq \mathcal{D}_{\boldsymbol{\alpha}, \bk} & \; \Longrightarrow \;  & \langle \bk\rangle^{-2} \, |\partial_t^2 \mathbb{P}|^2  \lsim \langle \bk\rangle^{-2} \; \mathcal{D}_{\boldsymbol{\alpha}, \bk}  \leq \langle \bk\rangle^{2} \; \big( \mathcal{D}_{\boldsymbol{\alpha}, \bk} \big)  &\;(\langle \bk\rangle \geq 1)
 \end{array}
		$$
After summation of the above three inequalities and by definition of $\mathcal{D}^{(2)}_{\boldsymbol{\alpha}, \bk}$, we deduce from \eqref{estdtE} that
\begin{equation} \label{dE}
\langle \bk\rangle^{-2}  \, |\partial_t \mathbb{E}|^2 \lsim \langle \bk\rangle^{2} \,  \mathcal{D}^{(2)}_{\boldsymbol{\alpha}, \bk}.
 \end{equation}
Analogously  from differentiating \eqref{M Lorentz Fourier} with respect to time, we get $\langle \bk\rangle^{-2}  \, |\partial_t \mathbb{H}|^2 \lsim \langle \bk\rangle^{2} \,  \mathcal{D}^{(2)}_{\boldsymbol{\alpha}, \bk}$ which, combined with \eqref{dE} and the definition \eqref{defenergy_densities1L} of $\mathcal{E}_\bk^1$ , leads to
\begin{equation} \label{densenergy1}
\langle \bk\rangle^{-2}  \, \mathcal{E}_\bk^{1} \lsim \langle \bk\rangle^{2} \,  \mathcal{D}^{(2)}_{\boldsymbol{\alpha}, \bk}.
\end{equation}
(b) Case $j=0$ : control of $\mathcal{E} _\bk$. \\ [12pt]Using the divergence free property \eqref{free divergenceLorentz} (i.e. $\bk\cdot \mathbb{H}=0$),  \eqref{E Lorentz Fourier} and \eqref{dE}, we get
$$|\bk \times \mathbb{H}  |^2=|\bk|^2\,|\mathbb{H}|^2\lsim |\partial_t\,\mathbb{E}|^2+|\partial_t \mathbb{P}|^2 \lsim    \langle \bk\rangle^{4} \; \mathcal{D}_{\boldsymbol{\alpha}, \bk}^{(2)}+ \mathcal{D}_{\boldsymbol{\alpha}, \bk} \lsim  \langle \bk\rangle^4 \; \mathcal{D}_{\boldsymbol{\alpha}, \bk}^{(2)}$$
since   $\mathcal{D}_{\boldsymbol{\alpha}, \bk} \leq \mathcal{D}_{\boldsymbol{\alpha}, \bk}^{(2)}$ and $1 \leq  \langle \bk\rangle^{4}.$ Thus $ |\mathbb{H}|^2\lsim |\bk|^{-2} \, \langle \bk\rangle^4 \;\mathcal{D}_{\boldsymbol{\alpha}, \bk}^{(2)}$, or equivalently
 \begin{equation}\label{H}
 |\mathbb{H}|^2\lsim  \big(\langle \bk\rangle^2+ |\bk|^{-2}\big)  \; \mathcal{D}_{\boldsymbol{\alpha}, \bk}^{(2)}.
 \end{equation}
Similarly, from \eqref{H Lorentz Fourier}  we obtain  $|\mathbb{E}|^2\lsim   \big(\langle \bk\rangle^2+ |\bk|^{-2}\big)  \; \mathcal{D}_{\boldsymbol{\alpha}, \bk}^{(2)}$ which, combined with \eqref{H} and the definition \eqref{defenergy_densitiesL} of $\mathcal{E}_\bk$, leads to
 \begin{equation} \label{densenergy0}
 \mathcal{E}_\bk \lsim  \big(\langle \bk\rangle^2+ |\bk|^{-2}\big) \,  \mathcal{D}^{(2)}_{\boldsymbol{\alpha}, \bk}.
 \end{equation}
(c) Case $j=2$ : control of $\langle \bk\rangle^{-4} \, \mathcal{E}^2 _\bk$. \\ [12pt]
Differentiating \eqref{E Lorentz Fourier} in time and then using \eqref{densenergy1}, we get
\begin{equation*}
|\partial_t^2\,\mathbb{E}|^2\lsim |\bk|^2\,|\partial_t\,\mathbb{H}|^2+ |\partial_t^2\,\mathbb{P}|^2\lsim 
\langle \bk\rangle^6 \; \mathcal{D}_{\boldsymbol{\alpha}, \bk}^{(2)} + \; \mathcal{D}_{\boldsymbol{\alpha}, \bk}^{1}.
\end{equation*}
Thus, as $\mathcal{D}_{\boldsymbol{\alpha}, \bk}^{1} \leq \langle \bk\rangle^2 \; \mathcal{D}_{\boldsymbol{\alpha}, \bk}^{(2)}\, \leq \langle \bk\rangle^6 \; \mathcal{D}_{\boldsymbol{\alpha}, \bk}^{(2)}$,  it follows that:
\begin{equation} \label{d2E}
\langle \bk\rangle^{-4} \, |\partial_t^2\,\mathbb{E}|^2 \lesssim \langle \bk\rangle^2 \; \mathcal{D}_{\boldsymbol{\alpha}, \bk}^{(2)} .
\end{equation}
Analogously, from \eqref{H Lorentz Fourier}, we get $\langle \bk\rangle^{-4} \, |\partial_t^2\,\mathbb{H}|^2 \lesssim \langle \bk\rangle^2 \; \mathcal{D}_{\boldsymbol{\alpha}, \bk}^{(2)}$ which, combined with \eqref{d2E} gives 
\begin{equation} \label{densenergy2}
\langle \bk\rangle^{-4}  \, \mathcal{E}_\bk^{2} \lsim \langle \bk\rangle^{2} \,  \mathcal{D}^{(2)}_{\boldsymbol{\alpha}, \bk}.
\end{equation}
Using (\ref{densenergy0}, \ref{densenergy1}, \ref{densenergy2}) in the definition (\ref{defcumenergy_densitiesL}) of the cumulated energy density $\mathcal{E}_\bk^{(2)}$, we get
\begin{equation}  \label{estcumdens}
\mathcal{E}_\bk^{(2)} \lsim \big(\langle \bk\rangle^2+ |\bk|^{-2}\big) \; \mathcal{D}_{\boldsymbol{\alpha}, \bk}^{(2)} 
\end{equation} 
 {\it Step 2: Control of the energy density $\mathcal{E}_{\boldsymbol{\Omega}, \bk}^{(2)}$.} We immediately observe from \eqref{defenergy_densitiesL},  \eqref{defdecay_densityL} (def. of $\mathcal{L}_\bk$ and $\mathcal{D}_\bk$), \eqref{defenergy_densities1L}, \eqref{defdecay_density1L}  (def. of $\mathcal{L}_\bk^1$ and $\mathcal{D}_\bk^1$), \eqref{defenergy_densities2L} and \eqref{defdecay_density2L}  
 (def. of $\mathcal{L}_\bk^2$ and $\mathcal{D}_\bk^2$) that
 $$
\mathcal{E}_{\boldsymbol{\Omega}, \bk} \lsim |\mathbb{P}|^2 + |\mathbb{M}|^2 +  \mathcal{D}_{\boldsymbol{\alpha}, \bk}, \quad \mathcal{E}_{\boldsymbol{\Omega}, \bk}^{1} \lsim \mathcal{D}_{\boldsymbol{\alpha}, \bk} +  \mathcal{D}_{\boldsymbol{\alpha}, \bk}^{1}, \quad \mathcal{E}_{\boldsymbol{\Omega}, \bk}^{2} \lsim \mathcal{D}_{\boldsymbol{\alpha}, \bk}^{1} +  \mathcal{D}_{\boldsymbol{\alpha}, \bk}^{2}
 $$
 thus, with the adequate linear combination, $\mathcal{E}_{\boldsymbol{\Omega}, \bk} ^{(2)} \lsim   \mathcal{D}_{\boldsymbol{\alpha}, \bk}^{(2)}  + |\mathbb{P}|^2 + |\mathbb{M}|^2 + \langle \bk\rangle^{-2} \, \mathcal{D}_{\boldsymbol{\alpha}, \bk}+ \langle \bk\rangle^{-4} \, \mathcal{D}_{\boldsymbol{\alpha}, \bk}^1$, i.e.
  \begin{equation} \label{add1}
 \mathcal{E}_{\boldsymbol{\Omega}, \bk} ^{(2)} \lsim   \mathcal{D}_{\boldsymbol{\alpha}, \bk}^{(2)}  + |\mathbb{P}|^2 + |\mathbb{M}|^2 .
 \end{equation}
It remains to control $\mathbb{P}$ and $\mathbb{M}$ which corresponds to the point (d) above. Using the constitutive equation \eqref{P Lorentz Fourier} and the definitions (\ref{defdecay_densityL}, \ref{defdecay_density1L}) of $(\mathcal{D}_{\boldsymbol{\alpha}, \bk}, \mathcal{D}_{\boldsymbol{\alpha}, \bk}^1)$, together with \eqref{densenergy0}, yields to
$$
|\mathbb{P}|^2 \lsim  |\partial_t\,\mathbb{P}|^2+|\partial_t^2\,\mathbb{P}|^2+|\mathbb{E}|^2\lsim  \mathcal{D}_{\boldsymbol{\alpha}, \bk} + \mathcal{D}_{\boldsymbol{\alpha}, \bk}^1 + \big(\langle \bk\rangle^2+ |\bk|^{-2}\big)  \; \mathcal{D}_{\boldsymbol{\alpha}, \bk}^{(2)},
$$
thus, as $\mathcal{D}_{\boldsymbol{\alpha}, \bk} \leq \mathcal{D}_{\boldsymbol{\alpha}, \bk}^{(2)}$ and $\mathcal{D}_{\boldsymbol{\alpha}, \bk}^1  \leq \langle \bk\rangle^2 \, \mathcal{D}_{\boldsymbol{\alpha}, \bk}^{(2)}$, (see \eqref{defcumdecay_densityL}),
\begin{equation}\label{P}
|\mathbb{P}|^2 \lsim   \big(\langle \bk\rangle^2+ |\bk|^{-2}\big)  \; \mathcal{D}_{\boldsymbol{\alpha}, \bk}^{(2)}.
\end{equation}
Likewise from \eqref{M Lorentz Fourier}, we have
\begin{equation}\label{M}
|\mathbb{M}|^2   \lsim   \big(\langle \bk\rangle^2+ |\bk|^{-2}\big)  \; \mathcal{D}_{\boldsymbol{\alpha}, \bk}^{(2)}.
\end{equation}
Using \eqref{P} and \eqref{M} in \eqref{add1}, we finally obtained
 \begin{equation} \label{add2}
\mathcal{E}_{\boldsymbol{\Omega}, \bk} ^{(2)} \lsim   \big(\langle \bk\rangle^2+ |\bk|^{-2}\big)  \; \mathcal{D}_{\boldsymbol{\alpha}, \bk}^{(2)}.
\end{equation}
The announced estimate \eqref{estimacion chidaL} follows from \eqref{estcumdens}, \eqref{add1} and the definition of $\mathcal{L}_\bk^{(2)}$.
 \end{proof}
\noindent Proceeding as in section \ref{Drude} for obtaining \eqref{eq.gronwall}, we deduce that, for some constant $\sigma > 0$, 
\begin{equation}\label{eq.gronwallL}
\forall \; \bk \neq 0, \quad \mathcal{L} _\bk^{(2)}(t)\leq \mathcal{L} _\bk^{(2)}(0) \; e^{- \frac{\sigma \, t}{\langle \bk\rangle^2+ |\bk|^{-2}}}.
\end{equation}
We shall use an estimate for the initial value $\mathcal{L} _\bk^{(2)}(0)$ which is the equivalent of the estimate for the Drude case. As getting this estimate is slightly more lengthy and tedious than for the Drude case, we give it in a lemma whose proof is delayed in the appendix, section \ref{sec-app-lemme-IC}.
\begin{lemma} \label{lem_estiL2initial}
	\begin{equation}\label{estimL2initial}
	\mathcal{L} _\bk^{(2)}(0)\lsim |\mathbb{E}_0(\bk)|^2+|\mathbb{H}_0(\bk)|^2.
	\end{equation}
	\end{lemma}
\noindent Therefore, using \eqref{estimL2initial} in \eqref{eq.gronwallL}, and as $\mathcal{L} _\bk \leq \mathcal{L} _\bk^{(2)}$, we have
\begin{equation} \label{estiLkL}
\mathcal{L} _\bk(t)\leq \big(|\mathbb{E}(\bk,0)|^2+|\mathbb{H}(\bk,0)|^2\big) \;   e^{- \frac{\sigma \, t}{\langle \bk\rangle^2+ |\bk|^{-2}}}.
\end{equation}

\noindent The exponential decay rate in \eqref{eq.gronwallL} degenerates when $|\bf k| \rightarrow + \infty$, as in the Drude case, but also  when $|\bf k| \rightarrow 0$ reason why the low (spatial) frequencies will need a special treatment leading to new assumptions involving the moments of the initial data. For this reason, 
we introduce the spaces, for $p \in \N$,
 \begin{equation} \label{spacesL1m} 
 \left\{ \begin{array}{l}
 \displaystyle L^1_{p}(\R^d) := \big\{ u \in L^1(\R^d) \; / \; (1+|x|)^p \, u \in L^1(\R^d) \big\} \\[12pt]
\displaystyle L^1_{p,0}(\R^d) := \Big\{ u \in L^1_{p}(\R^d) \; / \; \forall\;  \alpha \mbox{ such that } |\alpha| \leq p-1, 
 \int_{\R^d} x^\alpha \, u(x) \rmd x= 0 \Big\} \end{array}\right.
 \end{equation}
where  $\alpha = (\alpha_1, \alpha_2, \alpha_3)$ denotes a multi-index with ``length" $|\alpha| = \alpha_1+ \alpha_2+ \alpha_3$ and 
 \begin{equation} \label{notalpha}
 \partial^\alpha := \partial_1^{\alpha_1} \partial_2^{\alpha_2} \partial_1^{\alpha_3} \quad \mbox{and} \quad x^\alpha := x_1^{\alpha_1} x_2^{\alpha_2} x_3^{\alpha_3} \mbox{ for any }
 x = (x_1, x_2, x_3) \in \mathbb{R}^3.
 \end{equation}
{We point out that, from the definition \eqref{spacesL1m}, one has $L^1_{0}(\R^d)=L^1_{0,0}(\R^d)=L^1(\R^d)$.}
\\ [12pt]
\noindent We are now in position to state the main theorem which is expressed in terms of the augmented  energy $\mathcal{L}(t)$ that naturally replaces the one defined by \eqref{total_energy} for the Drude case, namely
\begin{equation}\label{total_energyL}
\left|\begin{array}{lll} \mathcal{L}(t) &  := &  \displaystyle  \mathcal{E}(\mathbf{E}, \mathbf{H}, t)+\frac{1}{2}\Big(\varepsilon_0\,\sum_{j=1}^{N_e} \Omega_{e,j}^2\,\sn{\partial_t\mathbf{P}_j(\cdot,t)}{L^2(\R^3)}^2+\mu_0\,\,\sum_{\ell=1}^{N_m} \Omega_{m,\ell}^2\,\sn{\partial_t\mathbf{M}_\ell(\cdot,t)}{L^2(\R^3)}^2\Big),  \\ [12pt] & + & \displaystyle\frac{1}{2}\Big(\varepsilon_0\,\sum_{j=1}^{N_e} \omega_{e,j}^2 \, \Omega_{e,j}^2\,\sn{\mathbf{P}_j(\cdot,t)}{L^2(\R^3)}^2+\mu_0\,\,\sum_{\ell=1}^{N_m} \omega_{m,\ell}^2\,\Omega_{m,\ell}^2\,\sn{\mathbf{M}_\ell(\cdot,t)}{L^2(\R^3)}^2\Big). 
\end{array} \right.
\end{equation}
\begin{theorem} \label{thm_Lorentz}
	For any $(\mathbf{E}_0, \mathbf{H}_0) \in L^2(\R^3)^3 \times L^2(\R^3)^3$ satisfying the free divergence condition, the total energy tends to $0$ when $t$ tends to $+\infty :$
	\begin{equation} \label{convergenceL}
	\lim_{t \rightarrow + \infty} \mathcal{L}(t) = 0.
	\end{equation}
	Moreover if for some integers $m \geq 0$ and $p \geq 0$,
	\begin{equation}\label{hypICL}
	\begin{array}{l}
	(\mathbf{E}_0, \mathbf{H}_0) \in H^{m}(\R^3)^3 \times  H^{m}(\R^3)^3, \quad 
	(\mathbf{E}_0, \mathbf{H}_0) \in L^1_{p,0}(\R^3)^3 \times  L^1_{p,0}(\R^3)^3,
	\end{array}
	\end{equation} 
	one has a polynomial decay rate
	\begin{equation} \label{polynomial_decay2}
	\mathcal{L}(t) \leq \frac{C_{\mathrm{HF}}^m(\mathbf{E}_0, \mathbf{H}_0)}{t^m} + \frac{C_{\mathrm{LF}}^p(\mathbf{E}_0, \mathbf{H}_0)}{t^{p+\frac{d}{2}}}
	\end{equation}
	where the above constants satisfy
	$$\left\{\begin{array}{l}\displaystyle C_{\mathrm{HF}}^m(\mathbf{E}_0, \mathbf{H}_0) \lsim \sn{\mathbf{E}_0}{H^{m}(\R^3)}^2 + \sn{\mathbf{H}_0}{H^{m}(\R^3)}^2, \\ [12pt]
	\displaystyle C_{\mathrm{LF}}^p (\mathbf{E}_0, \mathbf{H}_0)\lsim \sup_{|\alpha|=p} \big \|x^\alpha \mathbf{E}_0 \big\|_{L^1} + \sup_{|\alpha|=p} \big \|x^\alpha \mathbf{H}_0 \big\|_{L^1}. \end{array} \right.$$
\end{theorem}

\begin{proof}
	From the respective definitions of 	$\mathcal{L}$ and $\mathcal{L}_\bk$, see \eqref{defenergy_densitiesL} and \eqref{total_energyL}, we can use  Plancherel's identity in \eqref{estiLkL} to obtain
	\begin{equation} \label{energy_fourier}
	\mathcal{L}(t)=\int_{\R^3} \mathcal{L}_\bk(t) \; \rmd \bk \lsim \int_{\R^3} \big( \, |\mathbb{E}_0(\bk)|^2 + |\mathbb{H}_0(\bk)|^2 \, \big) \;   e^{- \frac{\sigma \, t}{\langle \bk\rangle^2+ |\bk|^{-2}}} \;  \rmd \bk.
	\end{equation}
	Next, as announced before, we treat low and high (space) frequencies separately. We begin with high frequencies who can be treated as for the Drude model.\\ [12pt]
	(i) If $|\bk|\geq 1$, then $\langle \bk\rangle^2+ |\bk|^{-2} \leq \langle \bk\rangle^2+ 1 \leq 2\, \langle \bk\rangle^2$. Thus 
	$$
	\int_{|\bk| \geq 1} \mathcal{L}_\bk(t) \; \rmd \bk \lsim \int_{|\bk| \geq 1} \big( \, |\mathbb{E}_0(\bk)|^2 + |\mathbb{H}_0(\bk)|^2 \, \big) \;   e^{- \frac{\sigma \, t}{2 \langle \bk\rangle^2}} \;  \rmd \bk.
	$$
	Bounding, in the right hand side, the integral over $|\bk|\geq 1$ by the one over $\R^3$, we can proceed as in the proof of Theorem \ref{thm_Drude} (with $\sigma/2$ instead of $\sigma$, cf. \eqref{energy_fourier-drude}) to obtain 
	\begin{equation} \label{estint1L}
	\int_{|\bk| \geq 1} \mathcal{L}_\bk(t) \; \rmd \bk \lsim\left(\frac{2m}{\sigma\,e \, t}\right)^{m} \; \Big( \sn{\mathbf{E}_0}{H^{m}(\R^3)}^2 + \sn{\mathbf{H}_0}{H^{m}(\R^3)}^2 \Big).
	\end{equation}
	(ii) If $|\bk|\leq 1$, then $\langle \bk\rangle^2+ |\bk|^{-2} \leq 3 \; |\bk|^{-2} $, (both $1$ and $|\bk|^2$ are smaller than $|\bk|^{-2}$) thus 
	\begin{equation} \label{estint1}
	\int_{|\bk| \leq 1} \mathcal{L}_\bk(t)  \, dk\lsim \int_{\R^3}  \Big( |\mathbb{E}_0(\bk)|^2 + |\mathbb{H}_0(\bk)|^2\Big) \;  e^{- \frac{|\bk|^2 \sigma t}{3}} \; \rmd \bk.
	\end{equation}
	The behaviour of the right hand side is obviously dominated by what happens when $|\bk|$ tends to $0$.\\[12pt]
	The condition $(\mathbf{E}_0, \mathbf{H}_0) \in L^1_{p,0}(\R^3)^3 \times  L^1_{p,0}(\R^3)^3$ implies in particular that
	$$
	\forall \; \alpha \; / \; |\alpha| \leq p-1, \quad \partial^\alpha \mathbb{E}_0(0) = \partial^\alpha \mathbb{H}_0(0).
	$$
Furthermore, {as $(\mathbf{E}_0, \mathbf{H}_0) \in L^1_{p,0}(\R^3)^3$, their Fourier transform $\mathbb{E}_0$ and $\mathbb{H}_0$ are bounded functions of class $\mathcal{C}^p$ whose partial derivatives are bounded up to the order $p$.}	Consequently, using a Taylor expansion at $0$ truncated  at order $p$, we have 
	$$
	|\mathbb{E}_0(\bk)| \lsim |\bk|^p \, \sup_{|\alpha|=p} \|\partial_\alpha \mathbb{E}_0\|_{L^\infty}, \quad |\mathbb{H}_0(\bk)| \lsim |\bk|^p \, \sup_{|\alpha|=p} \|\partial_\alpha \mathbb{H}_0\|_{L^\infty}
	$$
	which implies, using well known properties of the Fourier transform,
	$$
	|\mathbb{E}_0(\bk)| \lsim |\bk|^p \, \sup_{|\alpha|=p} \big \||x|^\alpha \mathbf{E}_0 \big\|_{L^1}, \quad |\mathbb{H}_0(\bk)| \lsim |\bk|^p \, \sup_{|\alpha|=p} \big \||x|^\alpha \mathbf{H}_0 \big\|_{L^1}
	$$
	Substituting the above in \eqref{estint1} yields
	$$
	\int_{|\bk| \leq 1}  \mathcal{L}_\bk(t) \, \rmd \bk.\lsim \Big( \int_{\R^d} |\bk|^{2p}\, e^{- \frac{2 |\bk|^2 \sigma t}{3}} \; dk \Big) 
	\Big(\sup_{|\alpha|=p} \big \|x^\alpha \mathbf{E}_0 \big\|_{L^1} + \sup_{|\alpha|=p} \big \|x^\alpha \mathbf{H}_0 \big\|_{L^1} \Big).
	$$
	With the change of variable $\mathbf{\xi }= \sqrt{\sigma t} \; \bk$, we compute that, for some constant $C(p,d) > 0$,
	$$
	\int_{\R^d} |\bk|^{2p}\, e^{- \frac{2 |\bk|^2 \sigma t}{3}} \; \rmd \bk = 
	\frac{1}{(\sigma t)^{p + \frac{d}{2}}} \; \int_{\R^d} |\xi|^{2p}\, e^{- \frac{2 |\xi|^2 }{3}} \; \rmd \mathbf{\xi } \equiv \frac{C(p,d)}{(\sigma t)^{p + \frac{d}{2}}} \; .
	$$
	Finally, with another constant $C$ that only depends on $d$, $p$ and the parameters of the model, we get
	\begin{equation} \label{estint2L}
	\int_{|\bk| \leq 1}  \mathcal{L}_\bk(t)  \, dk\lsim \frac{C}{t^{m + \frac{d}{2}}}
	\Big(\sup_{|\alpha|=m} \big \|x^\alpha \mathbf{E}_0 \big\|_{L^1} + \sup_{|\alpha|=m} \big \|x^\alpha \mathbf{H}_0 \big\|_{L^1} \Big).
	\end{equation}
	At the end, the final estimate \eqref{polynomial_decay2} is obtained by joining \eqref{estint1L} and \eqref{estint2L}.
\end{proof}
\begin{rem} [Comparison with the estimates of \cite{Nicaise2020}] The reader will check that, if we drop the second term in the estimate \eqref{polynomial_decay2}, which is specific to the problem in the whole space, our results for $m=1$ coincide qualitatively with the ones of \cite{Nicaise2020}, cf. \eqref{NicaiseEstimates} with $q=2$).
	\end{rem} 
\section{Extensions}\label{sec-Extensions-result}
\subsection{The problem in a bounded domain} \label{bounded_domains}
One can consider the evolution problem associated to equations \eqref{planteamiento Lorentz}  (or \eqref{planteamiento}) but posed in a bounded Lipschitz domain $\Omega$ of $\R^3$ and completed, for instance, with perfectly conducting conditions  (in some sense the ``Dirichlet" problem for Maxwell's equations)
\begin{equation}
{\bf E} \times {\bf n} = 0 \quad \mbox{ on } \partial \Omega, \quad (\mbox{with $\bf n$ the unit normal vector to $\partial \Omega $)},
\end{equation}
{where $\partial \Omega$ denotes the boundary of $\Omega$ and ${\bf n}$ the unit outward normal vector of $\partial \Omega$.}
In such a case, the analysis of sections \ref{Drude} and \ref{Lorentz_model} can be extended. The main difference is the the use of the Fourier transform in 
space has to be replaced by an adequate modal expansion. More precisely, we introduce the cavity eigenvalue problem:
\begin{equation} \label{modal_problem}
\left\{ \begin{array}{ll}
\mbox{Find } k \in \R \mbox{ and } ({\bf u}, {\bf v} ) \neq 0 \in L^2(\Omega)^3 \times L^2(\Omega)^3 \mbox{ such that} \\ [8pt]
\mathrm{i}\begin{pmatrix}  0 & \nabla\times\\ - \nabla\times & 0 \end{pmatrix}  \begin{pmatrix} \bf u \\ \bf v \end{pmatrix}  = k \; \begin{pmatrix} \bf u \\ \bf v \end{pmatrix}  \quad \mbox{in } \Omega, \\[12pt]
\nabla\cdot  {\bf u} =\nabla\cdot  {\bf v} = 0  \quad \mbox{in } \Omega, \quad  {\bf u} \times {\bf n}=  {\bf v} \cdot {\bf n}= 0
\quad \mbox{on }  \partial \Omega,
\end{array} \right.
\end{equation}
which corresponds to find the eigenvalues of the self-adjoint Maxwell operator $\mathcal{A}$ in the closed subspace of  $L^2(\Omega)^3$,  {$\mathcal{H} := \big\{ ({\bf u}, {\bf v})\in L^2(\Omega)^3 \, / \,\nabla\cdot  {\bf u} = \nabla\cdot  {\bf v} = 0 \ \mbox{ and } \  {\bf v}\cdot {\bf n}=0  \mbox{ on } \partial \Omega \}$,}  namely 
\begin{equation} \label{op_Maxwell}
\mathcal{A} \begin{pmatrix} \bf u \\ \bf v \end{pmatrix}  = \mathrm{i} \; \begin{pmatrix} \nabla \times {\bf v} \\ - \nabla \times \bf u \end{pmatrix} , \quad \forall \;  ({\bf u}, {\bf v})  \in D(\mathcal{A}) := \mathcal{H} \; \cap  \; \big(
{H_0(\mbox{rot};\Omega)\times H(\mbox{rot} ;\Omega) \big)},
\end{equation}
{where $H(\mbox{rot} ;\Omega):=\{ {\bf v }\in L^2(\Omega)^3 \mid \nabla \times {\bf v} \in L^2(\Omega)^3  \}$ and $H_0(\mbox{rot};\Omega):=\{ {\bf u}  \in H(\mbox{rot} ;\Omega)\mid {\bf u} \times {\bf n }=0 \mbox{ on }\partial \Omega \}$.}
{One show, see e.g.  \cite{Dau-1990} chapter IX, that  the operator $\mathcal{A}$ has a compact resolvent.}
From the theory of selfadjoint operators with compact resolvent, and using the symmetries of Maxwell's equations, one knows that there is a countable infinity of cavity modes 
\begin{equation} \label{cavity modes2}
( \pm \, k_n, {\bf u}_n^\pm,  {\bf v}_n^\pm) \in \R \times L^2(\Omega)^3 \times L^2(\Omega)^3 , \quad n \in \N, \quad \mbox{with } k_n \geq 0,\quad  k_n \rightarrow + \infty
\end{equation} 
with ${\bf u}_n^+= {\bf u}_n^-$ and ${\bf v}_n^+= - {\bf v}_n^-$ in such a way that 
$\big\{ ( {\bf u}_n^\pm,  {\bf v}_n^\pm)  , \; n \in \N \big\} $ form {an orthonormal basis of the Hilbert space  $\mathcal{H}$.}
Then, one can decompose the electromagnetic field as
$$
{\bf E}(\cdot,t) = \sum_\pm \sum_{n=0}^{+\infty}  \; \mathbb{E}_n^{\pm}(t) \;  {\bf u}_n^\pm, \quad  {\bf H}(\cdot,t) = \sum_\pm \sum_{n=0}^{+\infty}  \; \mathbb{H}_n^{\pm}(t) \;  {\bf v}_n^\pm,
$$
and the auxiliary field ${\bf P}_j$  and ${\bf M}_{\ell}$   (understood as in \eqref{rem_Drude}) accordingly 
$$
{\bf P}_j(\cdot,t) = \sum_\pm \sum_{n=0}^{+\infty}  \; \mathbb{P}_{j,n}^{\pm}(t) \;  {\bf u}_n^\pm, \quad  {\bf M}_{\ell}(\cdot,t) = \sum_\pm \sum_{n=0}^{+\infty}  \; \mathbb{M}_{\ell,n}^{\pm}(t) \; {\bf v}_n^\pm.
$$
The rest of the analysis follows exactly the same lines as for $\Omega = \R^3$, modulo the following substitutions
$$
{\bf k} \in \R^3 \; \rightarrow \; \{ (\pm, n), n \in \mathbb{N} \}, \quad \int_{\R^3} \rmd \bk \; \rightarrow \; \sum_\pm \sum_{n=0}^{+\infty} \;  , \quad \cdots \quad \mbox{etc.} 
$$
The results are then similar to the one of Theorem \ref{thm_Lorentz} {without the second term in the right hand side  of the estimate \eqref{polynomial_decay2} and the hypothesis of the second line of \eqref{hypICL}}, provided some modifications on the assumptions for the initial data which must  now satisfy 
\begin{equation}\label{hypICstat}
(\bE_0,\bH_0) \in \mathcal{H}_0 := (\mbox{Ker } \mathcal{A})^\perp, \quad  (\mbox{orthogonality in } \mathcal{H}) \quad(\mbox{see also remark } \ref{rem_stat})
\end{equation} 
and the Sobolev regularity \eqref{hypICL}(first line) must be replaced by
\begin{equation}\label{hypIC2}
(\mathbf{E}_0, \mathbf{H}_0) \in D(\mathcal{A}^m) \cap \mathcal{H}_0,  \quad(\mbox{see also remark } \ref{rem_reg}).
\end{equation} 
\begin{rem} \label{rem_stat}
	The (finitely dimensional) space $\mbox{Ker }\mathcal{A}$ is the space of electromagnetic static fields
	$$
	{\mbox{Ker } \mathcal{A} = \big\{ ({\bf u}, {\bf v}) \in L^2(\Omega)^3 \, / \, \nabla\cdot {\bf u} = \nabla\cdot  {\bf v} = 0, \nabla \times {\bf u} = \nabla \times {\bf v} = 0\mbox{ on }  \Omega  \mbox{ and }{\bf u} \times {\bf n}=0,  {\bf v} \cdot {\bf n} = 0  \mbox{ on } \partial \Omega \}.}
	$$
	When $\Omega$ in {\it simply connected}, it is known that $\mbox{Ker } \mathcal{A} = \{0\}$ and that its dimension increases with the complexity of the topology of $\Omega$. {For the proof of these assertions, we refer to \cite{Dau-1990}, chapter IX.}
	In some sense, the condition \eqref{hypICstat} can be seen as a substitute to the condition \eqref{hypICL}(second line).
	\end{rem}
{\begin{rem} \label{rem_reg}
The condition $(\mathbf{E}_0, \mathbf{H}_0) \in D(\mathcal{A}^m)$ implies $(\mathbf{E}_0, \mathbf{H}_0) \in H^{m}_{loc}(\Omega)^3 \times  H^{m}_{loc}(\Omega)^3$.
Furthermore, for $C^{\infty}$ domains $\Omega$, the condition $(\mathbf{E}_0, \mathbf{H}_0) \in D(\mathcal{A}^m)$ (see e.g. \cite{Dau-1990}, chapter IX) is equivalent to $(\mathbf{E}_0, \mathbf{H}_0) \in H^{m}(\Omega)^3 \times  H^{m}(\Omega)^3$  for any integer $m>0$.
\end{rem}}
\subsection{The case of mixed Drude-Lorentz models} \label{Drude_Lorentz}
{In the sums \eqref{epsmuLorentz} defining $\varepsilon(\omega)$ and $\mu(\omega)$ the resonance frequency  $\omega_{e,j}$ or $\omega_{m,\ell}$ are either  strictly positive or zero. For instance, for $\varepsilon(\omega)$, we shall say that }
\begin{equation} \label{DLterms}
\frac{\Omega_{e,j}^2}{\omega^2 + i \, \alpha_{e,j} \, \omega - \omega_{e,j}^2}, \; \omega_j > 0 \mbox{ is a Lorentz term,} \quad \quad  \frac{\Omega_{e,j}^2}{\omega^2 + i \, \alpha_{e,j} \, \omega } \mbox{ is a Drude term.}
\end{equation}
In section \ref{Drude} (standard Drude model), we consider the case where $\varepsilon(\omega)$ and $\mu(\omega)$ contained a single Drude term while
in section \ref{Lorentz_model}, we consider the case where $\varepsilon(\omega)$ and $\mu(\omega)$ only contained Lorentz terms, because of assumption 
(\ref{hypomegabis}, second line). \\[12pt]
It is natural to look at the cases where $\varepsilon(\omega)$ (resp. $\mu(\omega)$) contains Lorentz terms but also Drude terms whose number
is $N_{d,e}$ (respectively $N_{d,m}$). It appears that our Lyapunov approach can easily handle these cases (modulo minor additional manipulations) with the following results:
\begin{itemize}
	\item If $N_{d,e}> 0$ and $N_{d,m}>0$, the result is the same as for the Drude model (see theorem \ref{thm_Drude}).
	\item If $N_{d,e}> 0$ and $N_{d,m}=0$ (or the contrary), the result is the one for Lorentz (see theorem \ref{thm_Lorentz}). 
\end{itemize}
\appendix
\section{Appendix}\label{sec.appendix}

\subsection{On the dissipation condition of \cite{Figotin} for Lorentz models}
\label{sec-append-figotin}
Let us recall that, when the limit  \eqref{limits} exist almost everywhere on the real axis (which is the case for Lorentz models), the  sufficient dissipation  condition (6.4)  given in  \cite{Figotin} reads  
\begin{equation} \label{dissFigotinappendix}
\mbox{for a. e. }\omega \in \R, \quad \Imag \, \omega \, \hat{\chi}_e (\omega) \geq \gamma(\omega)^{-1} > 0, \quad \gamma\in L^1_{loc}(\R).
\end{equation}
By virtue of the expression \eqref{epsmuLorentz}(a)  of the complex permittivity $\varepsilon(\omega)$ and the formula \eqref{eq.complexperm}, we have, as soon as {$\omega\in \R$ and $\omega^2 + i \, \alpha_{e,j} \, \omega - \omega_{e,j}^2 \neq 0$ for all $j\in \{ 1,\ldots,N_e\}$:}
$$
\Imag \, \omega \, \hat \chi_e(\omega) =   \sum_{j=1}^{N_e} 
\frac{ \alpha_{e,j} \,  \Omega_{e,j}^2 \; \omega^2}
{|\omega^2 + i \, \alpha_{e,j} \, \omega - \omega_{e,j}^2|^2}.
$$
Let $J _0:= \{ j\in \{1,\ldots, N_e\} \; / \; \omega_{e,j} = 0\}$ and  $J_+:= \{ j \in \{ 1,\ldots, N_e\}\; / \; \alpha_{e,j} > 0\}$. 
\begin{itemize}
	\item  [(i)]$J_+ = \emptyset$, i. e.  for all $j \in \{1, \ldots, N_e \}$ , $\alpha_{e,j}=0$
	then $\Imag \, \omega \, \hat \chi_e(\omega)=0$
and \eqref{dissFigotin}(i) can not hold.
\item  [(ii)] $J_+ \neq \emptyset$. We distinguish two subcases  
\begin{itemize} \item[(a)] $J_+ \cap J_0 \neq \emptyset$. This means 
that one $\omega_{e,j}$ vanishes, for instance $\omega_{e,1} = 0$, and  $\alpha_{e,1} > 0$. Then
$$ \ds
\mathcal{I}m \, \omega \, \hat \chi_e(\omega) = \frac{ \alpha_{e,1} \,  \Omega_{e,1}^2 \; \omega^2}
{|\omega^2 + i \, \alpha_{e,j} \, \omega |^2}  +  \sum_{j=2}^{N_e} 
\frac{ \alpha_{e,j} \,  \Omega_{e,j}^2 \; \omega^2}
{|\omega^2 + i \, \alpha_{e,j} \, \omega - \omega_{e,j}^2|^2} \geq  \frac{\alpha_{e,1} \,  \Omega_{e,1}^2}{|\omega+\rmi \alpha_{e,1}| }
$$
in which case \eqref{dissFigotin}(ii) holds true with {$\gamma: \omega \mapsto \big(\alpha_{e,1} \,  \Omega_{e,1}^2\big)^{-1}\, |\omega+\rmi \alpha_{e,1}|\in L^1_{loc}(\mathbb{R})$. }\\
\item [(b)] $J_+ \cap J_0= \emptyset$. 
In this case
$
\ds \Imag\, \omega \, \hat \chi_e(\omega) \underset{\omega \rightarrow 0} \sim  
\Big( \sum_{j\in J_+}  \frac{ \alpha_{e,j} \,  \Omega_{e,j}^2}{\omega_{e,j}^4}\Big)  \; \omega^2,
$
and \eqref{dissFigotin}(i) can not hold.
\end{itemize} 
\end{itemize}
\subsection{On the energy indentity \eqref{Identitegenerale}}\label{sec-app-Lya}
{Let $\bE,\bH\in C^{0}\big(\R^+,H(\mbox{rot} ; \R^3)\big)\cap C^1\big(\R^+,L^2(\R^3)^3\big)$ and $\bE, \, \bH, \, \bD, \bB\in C^{1}(\R^+,L^2(\R^3)^3)$} be solutions of the equations \eqref{maxwell3D}, \eqref{CL1} and \eqref{CL2} where $\chi_\nu\in C^3(\mathbb{R}^+) $ for $\nu=e,m$. We explain in this appendix how to obtain the energy identity \eqref{Identitegenerale} with \eqref{Lyapunovgeneral} and \eqref{Dissipationfunction}. \\ [12pt]
From equations \eqref{maxwell3D}, \eqref{CL1} and \eqref{CL2} \, it is straightforward to deduce the identity (with $\mathcal{E}(t)$ the electromagnetic energy, see \eqref{def_energyEM0}) 
\begin{equation}\label{eq.energymax}
\frac{\rmd \mathcal{E}(t)} {\rmd t}+  \mathcal{I}(t)  = 0, \quad \mathcal{I}(t) := \int_{\R^3}  \partial_t \textbf{P}_{\mathrm{tot}}(\bx,t)  \,  \bE(\bx,t) \rmd \bx+  \int_{\R^3}  \partial_t \textbf{M}_{\mathrm{tot}}(\bx,t) \,  \bH (\bx,t)    \rmd \bx.
\end{equation}
Differentiating the constitutive laws \eqref{CL2} in time, we have
$$
\begin{array}{l}
\ds \partial_t\textbf{P}_{\mathrm{tot}}(\bx,t)=\varepsilon_0 \int_{0}^t \chi'_{e}(t-s)\, \bE(\bx,s) \; \rmd s + \varepsilon_0 \, \chi_e(0) \, \bE(\bx,t) \\ [12pt]
\ds \partial_t\textbf{M}_{\mathrm{tot}}(\bx,t)=\mu_0 \int_{0}^t \chi'_{m}(t) \, \bH(\bx,t-s) \; \rmd s + \mu_0 \, \chi_m(0) \, \bH(\bx, t)
\end{array}
$$
from which we deduce that
\begin{equation} \label{expI}
\left| \begin{array}{l}
\ds \mathcal{I}(t) =  \mathcal{I}_e(t) + \mathcal{I}_m(t) \quad \mbox{with} \\ [12pt]
\ds \mathcal{I}_e(t) := \varepsilon_0 \, \chi_e(0) \int_{\R^3} |\bE(\bx,t)|^2 \rmd \bx + \varepsilon_0 \, \int_{\R^3} \Big( \bE(\bx,t)  \cdot \int_0^t \chi'_{e}(t-s)\, \bE(\bx,s) \; \rmd s  \Big) \rmd \bx, \\ [12pt]
\ds \mathcal{I}_m(t) := \mu_0 \, \chi_m(0) \int_{\R^3} |\bH(\bx,t)|^2\rmd \bx + \mu_0 \, \int_{\R^3} \Big( \bH(\bx,t)  \cdot \int_0^t \chi'_{m}(t-s)\, \bH(\bx,s) \; \rmd s  \Big) \rmd \bx. 
\end{array} \right.
\end{equation}
It remains to transform $\mathcal{I}_e(t)$ and $\mathcal{I}_m(t)$.\\ [12pt]
 The basic technical ingredient concerns convolution type quadratic forms:
\begin{equation} \label{QCF}
Qu(t) := \int_0^t k'(t-s) \, u(s) u'(t) \, ds,
	\end{equation}
where $k(t)$ is a given convolution kernel.  When $k(t) = C \, \delta(t)$, formally $Qu(t) = C \, \frac{d}{dt} |u(t)|^2 $. The next lemma (lemma 3.2 in \cite{Riv-04}) generalizes this observation to a smooth kernel. For completeness, we provide here a constructive proof (which is not given in \cite{Riv-04}).
\begin{lemma} \label{lemQCF} 
	Given $k \in C^2(\mathbb{R}^+)$ and  ${u \in C^1(\mathbb{R}^+)}$ with $u(0) = 0$, $Qu(t) $ given by \eqref{QCF} satisfies
\begin{equation} \label{propQCF}
{\left| \begin{array}{lll}
	Qu(t) & = & \ds \frac{1}{2} \; \frac{d}{dt} \Big[ \big( k(t) - k(0) \big) |u(t)|^2- \int_0^t k'(t-s) \,  \big(u(s)- u(t)\big)^2\,  \rmd s\Big] \\ [12pt]  
	&  &  \ds - \frac{1}{2} \,  k'(t) \, |u(t)|^2 +\frac{1}{2} \int_0^t k''(t-s) \,  \big(u(s)- u(t)\big)^2\, \rmd s.
\end{array} \right.}
\end{equation}	

\end{lemma}
\begin{proof}
The guiding idea is to make appear time derivatives of square quantities in the expression of $Qu(t)$. \\ [12pt]
The main trick is to write $u(s) = u(t) + \big(u(s) - u(t)\big)$ in \eqref{QCF}, in such a way that
\begin{equation} \label {exp0}
		Qu(t) = \frac{1}{2} \; \big( k(t) - k(0) \big) \; \frac{d}{dt} |u(t)|^2 + \int_0^t k'(t-s) \, \big(u(s)- u(t)\big)  u'(t) \,  \rmd s.
	\end{equation}
	On the one hand, one has
	\begin{equation} \label {exp1}
	\frac{1}{2} \; \big( k(t) - k(0) \big) \; \frac{d}{dt} |u(t)|^2= \frac{1}{2} \frac{d}{dt} \Big(\big( k(t) - k(0) \big) |u(t)|^2\Big) - \frac{1}{2} \; k'(t) |u(t)|^2 .
	\end{equation}
	On the other hand, observing that $\big(u(t)- u(s)\big) \, u'(t) = \frac{1}{2} \frac{d}{dt} \big[ \big(u(s)- u(t)\big)^2]$ we have
	$$
	\int_0^t k'(t-s) \, \big(u(s)- u(t)\big) \,  u'(t) \,   \rmd s = -  \frac{1}{2} \int_0^t k'(t-s) \,  \frac{d}{dt} \big[\big(u(s)- u(t)\big)^2 \, \big]\,  \rmd s,
	$$
i.e., since $k'(t-s) \,  \frac{d}{dt} \big[\big(u(s)- u(t)\big)^2 \, \big] = \frac{d}{dt} \big[ \, k'(t-s) \, \big(u(s)- u(t)\big)^2 \, \big] - k''(t-s) \, \big(u(s)- u(t)\big)^2$, 
	$$
	\left| \begin{array}{lll}
	\ds \int_0^t k'(t-s) \, \big(u(s)- u(t)\big)  u'(t) \, \rmd s & = & \ds -  \; \frac{1}{2} \int_0^t \frac{d}{dt}    \big[ k'(t-s) \, \big(u(s)- u(t)\big)^2 \, \big]\,  \rmd s   \\ [12pt] & + & \ds \frac{1}{2} \int_0^t k''(t-s) \,  \big(u(s)- u(t)\big)^2 \,  \rmd s .
	\end{array} \right. 
	$$
	Finally 
	$
	\ds \int_0^t \frac{d}{dt} \big[ \, k'(t-s) \, \big(u(s)- u(t)\big)^2 \, \big]  \rmd s ={ \frac{d}{dt} \int_0^t  k'(t-s) \, \big(u(s)- u(t)\big)^2 \,   \rmd s }
	$, thus
\begin{equation} \label{exp2} 
\left| \begin{array}{lll}
\ds \int_0^t k'(t-s) \, \big(u(s)- u(t)\big)  u'(t) \, \rmd s& = &  \ds -  \; \frac{1}{2} \frac{d}{dt} \int_0^t  k'(t-s) \, \big(u(s)- u(t)\big)^2 \,   \rmd s    \\ [12pt]  
&&+\ds \frac{1}{2} \int_0^t k''(t-s) \,  \big(u(s)- u(t)\big)^2 \,  \rmd s .
\end{array} \right.
\end{equation}
%
Finally, substituting \eqref{exp1} and \eqref{exp2} in \eqref{exp0} leads to \eqref{propQCF}.
\end{proof}
\noindent
Now, we wish to transform the integrand (in space) in the second term of the expression \eqref{expI} of $\mathcal{I}_e(t)$ by making appear a quantity of the form \eqref{QCF}.
For this, it is useful to introduce primitives (in time) of the fields. More precisely, 
for ${\bf F} =\bE, \bH,  \bD, \bB,  \bP,  \bM$, we define for $t\geq 0$
$$
{\bf F}_p(\bx,t)=\int_{0}^t f(\bx,s) \, \mathrm{d}s, \quad  \forall\;  t\geq 0 \mbox{ and } \ \mathrm{a. e.} \  \bx\in \mathbb{R}^3. 
$$ 
In order to transform $\mathcal{I}_e(t)$, we first perform an integration by parts in time to get, since $\bE_p(\bx,0) =0$,
$$
\int_0^t \chi'_{e}(t-s)\, \bE(\bx,s) \; \rmd s  \equiv \int_0^t \chi'_{e}(t-s)\, \partial_t \bE_p(\bx,s) \; \rmd s = \chi_e'(0) \, \bE_p(\bx,t)  + \int_0^t \chi''_{e}(t-s)\,  \bE_p(\bx,s) \; \rmd s.
$$
As a consequence, since $\bE = \partial_t \bE_p$, we have
$$
\bE(\bx,t)  \cdot \int_0^t \chi'_{e}(t-s)\, \bE(\bx,s) \; \rmd s = \frac{\chi_e'(0)}{2} \, \frac{d}{dt} \, \big|\bE_p(\bx,t)\big|^2 + \int_0^t \chi''_{e}(t-s)\,  \bE_p(\bx,s) \cdot \partial_t \bE_p (\bx,t)\; \rmd s .
$$
The second term in the right hand side of the above expression is a sum of terms of the form \eqref{QCF} with $k = \chi_e'$. Thus,  integrating in space and then using the lemma \ref{lemQCF}, and substituting the resulting equality in the expression \eqref{expI} of $\mathcal{I}_e(t)$ gives (note that the terms involving {$\chi_e'(0)$} cancel each other) 
\begin{equation} \label{expIe} 
\left| \begin{array}{lll}
\mathcal{I}_e(t) & = &\ds \frac{\varepsilon_0}{2} \; \frac{d}{dt}  \Big( \chi'_e(t) \int_{\R^3} |\bE_p(\bx,t)|^2 \,  \rmd \bx\Big)-\frac{\varepsilon_0}{2} \frac{d}{dt}  \Big(\int_{0}^{t} \chi_e''(t-s)  \Big(  \int_{\R^3} |\bE_p(\bx,t) - \bE_p(\bx,s) |^2 \,  \rmd \bx \Big) \, \rmd s \Big)    \\ [12pt]  
&  &+\; \ds { \varepsilon_0 \, \chi_e(0) \int_{\R^3} |\bE(\bx,t)|^2 \rmd \bx } -  \frac{\varepsilon_0}{2} \;  \chi''_e(t)  \int_{\R^3} |\bE_p(\bx,t)|^2 \,  \rmd \bx \\[12pt]
&&\ds+\, \frac{\varepsilon_0}{2} \int_0^t \chi'''_e(t-s)  \,  \Big(  \int_{\R^3} |\bE_p(\bx,t) - \bE_p(\bx,s) |^2 \,  \rmd \bx \Big) \, \rmd s.
\end{array} \right.
\end{equation}
Analogously, we have 
\begin{equation} \label{expIm} 
\left| \begin{array}{lll}
\mathcal{I}_m(t) & = &\ds \frac{\mu_0}{2} \; \frac{d}{dt}  \Big( \chi'_m(t) \int_{\R^3} |\bH_p(\bx,t)|^2 \,  \rmd \bx\Big)-\frac{\mu_0}{2} \frac{d}{dt}  \Big(\int_{0}^{t} \chi_m''(t-s)  \Big(  \int_{\R^3} |\bH_p(\bx,t) - \bH_p(\bx,s) |^2 \,  \rmd \bx \Big) \, \rmd s \Big)    \\ [12pt]  
&  &+ \, \ds { \mu_0 \, \chi_m(0) \int_{\R^3} |\bH(\bx,t)|^2 \rmd \bx } -  \frac{\mu_0}{2} \;  \chi''_m(t)  \int_{\R^3} |\bH_p(\bx,t)|^2 \,  \rmd \bx \\[12pt]
&&\ds+ \, \frac{\mu_0}{2} \int_0^t \chi'''_m(t-s)  \,  \Big(  \int_{\R^3} |\bH_p(\bx,t) - \bH_p(\bx,s) |^2 \,  \rmd \bx \Big) \, \rmd s.
\end{array} \right.
\end{equation}
Finally, \eqref{Identitegenerale}  is simply obtained by gathering \eqref{eq.energymax}, \eqref{expI}, \eqref{expIe} and \eqref{expIm}.
\subsection{Estimating $\mathcal{L} _\bk^{(2)}(0)$ in the Lorentz case}\label{sec-app-lemme-IC}
Below, we use the notation of Section \ref{Lorentz_model} and our goal is to proof the estimate of Lemma \ref{lem_estiL2initial}, namely
	\begin{equation}\label{estimL2initialbis}
	\mathcal{L} _\bk^{(2)}(0)\lsim |\mathbb{E}_0(\bk)|^2+|\mathbb{H}_0(\bk)|^2.
	\end{equation}
First, by definition of  $\mathcal{L} _\bk(0)$ and since $\mathbb{P}(\bk,0)=\mathbb{M}(\bk,0)=\partial_t \mathbb{P}(\bk,0)=\partial_t \mathbb{M}(\bk,0) = 0$,
\begin{equation}\label{L00}
\mathcal{L} _\bk(0)=\frac{1}{2}\, \big(\varepsilon_0\,|\mathbb{E}_0(\bk)|^2+\mu_0\,|\mathbb{H}_0(\bk)|^2\big).
\end{equation}
Next we estimate $
\langle \bk\rangle^{-2} \, \mathcal{L}^1 _\bk(0)  = \langle \bk\rangle^{-2} \, \mathcal{E}^1_\bk(0)   + \langle \bk\rangle^{-2} \, \mathcal{E}_{\boldsymbol{\Omega}, \bk}^1(0)  
$.\\ [12pt]
From equations \eqref{E Lorentz Fourier} and \eqref{H Lorentz Fourier} at $t=0$, since $\partial_t\mathbb{P}(\bk,0)=\partial_t \mathbb{M}(\bk,0) = 0$,
$$
\langle \bk\rangle^{-2} \, \mathcal{E}^1 _\bk(0) \lsim \langle \bk\rangle^{-2} \, \big( |\bk \times \mathbb{E}_0(\bk)|^2+|\bk \times \mathbb{H}_0(\bk)|^2 \big) \lsim  |\mathbb{E}_0(\bk)|^2+|\mathbb{H}_0(\bk)|^2.
$$
Since $\mathbb{P}(\bk,0)=\mathbb{M}(\bk,0)=\partial_t \mathbb{P}(\bk,0)=\partial_t \mathbb{M}(\bk,0) = 0$, from \eqref{P Lorentz Fourier} and \eqref{M Lorentz Fourier} at $t=0$, we have
$$
 \langle \bk\rangle^{-2} \, \mathcal{E}_{\boldsymbol{\Omega}, \bk}^1(0)  \lsim  \langle \bk\rangle^{-2} \big(|\partial_t^2 \mathbb{P}(\bk,0)|^2 + |\partial_t^2 \mathbb{M}(\bk,0)|^2\big)  \lsim    |\mathbb{E}_0(\bk)|^2+|\mathbb{H}_0(\bk)|^2
$$
and as a consequence 
 \begin{equation} \label{L10}
\langle \bk\rangle^{-2} \, \mathcal{L}^1 _\bk(0) \lsim  |\mathbb{E}_0(\bk)|^2+|\mathbb{H}_0(\bk)|^2.
\end{equation}
Finally we estimate $
\langle \bk\rangle^{-4} \, \mathcal{L}^2 _\bk(0)  = \langle \bk\rangle^{-4} \, \mathcal{E}^2_\bk(0)   + \langle \bk\rangle^{-4} \, \mathcal{E}_{\boldsymbol{\Omega}, \bk}^2(0)  
$.\\ [12pt]
For bounding $\langle \bk\rangle^{-4} \, \mathcal{E}^2_\bk(0)$, we differentiate \eqref{E Lorentz Fourier} (resp. \eqref{H Lorentz Fourier}) and evaluate the resulting equations  at $t=0$ to express $\partial_t^2 \mathbb{E}(\bk,0)$ (resp. {$\partial_t^2 \mathbb{H}(\bk,0)$) in terms of $\bk \times \partial_t\,\mathbb{H}(\bk,0)$ and  $\partial_t^2 \mathbb{P}(\bk,0)$  (resp.  $\bk \times \partial_t\,\mathbb{E}(\bk,0)$ and $\partial_t^2 \mathbb{M}(\bk,0)$)}. From this, we deduce 
\begin{equation*}
\langle \bk\rangle^{-4}  \, \mathcal{E}^2_\bk(0) \lesssim \langle \bk\rangle^{-4}  \,( |\bk|^2\,\mathcal{E}^1_\bk(0) + |\partial_t^2 \mathbb{P}(\bk,0)|^2+ \partial_t^2 \mathbb{M}(\bk,0)|^2)\lsim \langle \bk\rangle^{-2}  \,\mathcal{L}^1 _\bk(0) \lsim  |\mathbb{E}_0(\bk)|^2+|\mathbb{H}_0(\bk)|^2
\end{equation*}
where for the last inequality we have used \eqref{L10}. For the last term we first observe that 
$$
\langle \bk\rangle^{-4} \, \mathcal{E}_{\boldsymbol{\Omega}, \bk}^2(0)   \lsim \langle \bk\rangle^{-4}   (| \partial_t^3\mathbb{P}(\bk,0)|^2 + | \partial_t^3\mathbb{M}(\bk,0)|^2+  | \partial_t^2\mathbb{P}(\bk,0)|^2 + | \partial_t^2\mathbb{M}(\bk,0)|^2)$$ 
and we obtain $\partial_t^3\mathbb{P}(\bk,0)$ (resp. $\partial_t^3\mathbb{M}(\bk,0)$) in function of $\partial_t \mathbb{E}(\bk,0)$ and $\partial_t^2\mathbb{P}(\bk,0)$ (resp. $\partial_t \mathbb{H}(\bk,0)$ and $\partial_t^2\mathbb{M}(\bk,0)$)
from equations \eqref{P Lorentz Fourier} and \eqref{M Lorentz Fourier}  after time differentiation. This leads to
\begin{equation*}
\langle \bk\rangle^{-4}  \, \mathcal{E}_{\boldsymbol{\Omega}, \bk}^2(0) \lesssim \langle \bk\rangle^{-4}  \, \langle \bk\rangle^{2}\,\mathcal{L}^1_\bk(0) \lsim \langle \bk\rangle^{-2}  \,\mathcal{L}^1_\bk(0) \lsim  |\mathbb{E}_0(\bk)|^2+|\mathbb{H}_0(\bk)|^2
\end{equation*}
{where uses \eqref{L10} for the last inequality.}
Adding the last two inequalities we obtain
 \begin{equation} \label{L20}
\langle \bk\rangle^{-4} \, \mathcal{L}^2_\bk(0) \lsim  |\mathbb{E}_0(\bk)|^2+|\mathbb{H}_0(\bk)|^2.
\end{equation}
Finally. \eqref{estimL2initialbis} is deduced from \eqref{L00}, \eqref{L10} and \eqref{L20}.

\subsection{Well-posdness and regularity of the solutions of the Cauchy problem in generalized Lorentz media}\label{sec-apend-D}
The (Cauchy) problem \eqref{planteamiento Lorentz}  can be rewritten as a generalized Schr\"{o}dinger  evolution problem:
\begin{equation}\label{eq.schro}
\frac{\rmd \, \bU}{\rmd\, t} + \rmi\, \bbA \, \bU=0 \mbox{ with } \bU(0)=\bU_0,
\end{equation}
where  the Hamiltonian $\bbA$ is an unbounded operator on the Hilbert-space:
\begin{equation}\label{eqdefHilbertspace}
\mathcal{H}:=L^2(\mathbb{R}^3)^3\times L^2(\mathbb{R}^3)^3 \times L^2(\mathbb{R}^3)^{3N_e}
	\times L^2(\mathbb{R}^3)^{3N_e} \times L^2(\mathbb{R}^3)^{3N_m}  \times L^2(\mathbb{R}^3)^{3N_m} ,
\end{equation}
endowed by the following inner product : for any $\bU=(\bE, \bH, \bP, \dot{\bP},\bM , \dot{\bM}) $ and $\bU^{\prime}=(\bE^{\prime}, \bH^{\prime}, \bP^{\prime},\bM^{\prime}, \dot{\bP}^{\prime}, \dot{\bM}^{\prime})$, where $(\bP,\bM)$ is defined as in  \eqref{notPM}, (and the same for $(\bP^{\prime},\bM^{\prime})$, $(\dot{\bP}, \dot{\bM})$ and $(\dot{\bP}^{\prime}, \dot{\bM}^{\prime})$) 
\begin{equation*} \label{defPS}
\begin{array}{|lllll}
(\bU, \bU')_{\mathcal{H}} \! &  \! =  \! &  \!  \ds \frac{\varepsilon_0}{2}\ (\bE,\bE^{\prime})_{L^2} + \frac{\mu_0}{2}\ (\bH,\bH^{\prime})_{L^2} \!  \! &\! + \!&  \! \displaystyle\frac{ \varepsilon_0}{2}  \sum_{j=1}^{N_e} \omega_{e,j}^2\, \Omega_{e,j}^2\, ( \bP_{j},\bP_{j}^{\prime})_{L^2} +
\frac{ \varepsilon_0}{2}  \sum_{j=1}^{N_e} \Omega_{e,j}^2\, (\dot{\bP}_{j},\dot{\bP}_{j}^{\prime})_{L^2}  \\
&&  \!  \! &\! + \!&  \!  \ds \frac{\mu_0}{2} \sum_{\ell=1}^{N_m} \Omega_{e,\ell}^2\, (\bM_{\ell},\bM_{\ell}^{\prime})_{L^2}+ \frac{\mu_0}{2} \sum_{\ell=1}^{N_m}  \Omega_{m,\ell}^2\, ( \dot{\bM}_{\ell},\dot{\bM}_{\ell}^{\prime})_{L^2}\ ,
\end{array}
\end{equation*}
 More precisely, if we introduce  $\bbA: D(\bbA)\subset \mathcal{H}\to \mathcal{H}$ defined by 
\begin{equation}\label{eq.defHamil}
\forall \; \bU =(\bE, \bH, \bP, \dot{\bP},\bM , \dot{\bM})\in D(\bbA), \quad 	\bbA \bU := -\rmi \; 	\begin{pmatrix} \varepsilon_0^{-1} \nabla \times \bH+\sum \; \Omega_{e,j}^2 \bP_j \\[10pt] 
 -\mu_0^{-1} \nabla \times \bE+ \sum \; \Omega_{m,\ell}^2 \dot{\bM}_{\ell}  \\[8pt]- \dot{\bP} \\[6pt]
  \alpha_{e,j} \dot{\bP}_j+\omega_{e,j}^2\,\bP_j-\bE\\[6pt] 
 - \dot{\bM} \\[6pt]  
 \alpha_{m,\ell}\, \dot{\bM}_m+\omega_{m,\ell}^2\,\bM_\ell-\bH
	\end{pmatrix} 
\end{equation}
the domain $D(\bbA)$ (dense in $\mathcal{H}$) being given by 
\begin{equation}\label{eq.defDomain}
D(\bbA):=H(\mbox{rot};\R^3) \times H(\mbox{rot};\R^3) \times L^2(\mathbb{R}^3)^{3N_e}
\times L^2(\mathbb{R}^3)^{3N_e} \times L^2(\mathbb{R}^3)^{3N_m}  \times L^2(\mathbb{R}^3)^{3N_m} , \end{equation}
we observe that one can rewrite \eqref{planteamiento Lorentz}  as \eqref{eq.schro}  with  the 
initial condition $\bU_0=(\bE_0, \bH_0, 0, 0, 0,0)\in  \mathcal{H}$. \\ [12pt]
The well-posedness of \eqref{eq.schro}  is ensured by the following lemma:
\begin{lemma}
If the assumption \eqref{WD} holds, the operator $-\rmi \, \bbA$ is a maximal dissipative operator. 
\end{lemma}
\begin{proof}
To show that $-\rmi \, \bbA$ is a maximal dissipative operator (see \cite{Dau-1992}, theorem 8 p 340) is equivalent to show one one hand that $-\rmi \, \bbA $ is dissipative, i.e. $\operatorname{Im}(\bbA \bU , \bU)\leq 0$ for all $\bU\in D(\bbA)$ and that  there exists $\omega\in  \C^+$ such that  $\bbA-\omega  \, \mathrm{I}$ is surjective. Thus the proof will consists in two steps.\\[4pt]
\noindent {\bf Step 1: $-\rmi \,  \bbA$ is dissipative.}
Performing  $(\bbA \bU , \bU)_{\mathcal{H}}$ thanks to \eqref{defPS} and \eqref{eq.defHamil}, after integration by parts, i. e. s $(\nabla \times \bH, \bE)_{L^2}=(\bH, \nabla \times  \bE)_{L^2}=\overline{(\nabla \times  \bE, \bH)}_{L^2}$, one finds that
$$
\operatorname{Im}(\bbA \bU , \bU) =  -\sum_{j=1}^{N_e} \alpha_{e,j} \, \Omega_{e,j}^2  \, \|\dot{\bP}_j\|_{L^2}^2 - \sum_{\ell=1}^{N_m} \alpha_{m,\ell} \, \Omega_{m,\ell}^2  \, \|\dot{\bM}_\ell\|_{L^2}^2\leq 0.
$$
\noindent {\bf Step 2: for any $\omega\in  \C^+$,  $(\bbA-\omega  \mathrm{I}) D(\bbA)= \mathcal{H}$.}
We prove also the injectivity of the operator $\bbA-\omega  \mathrm{I}$ by showing that for any $\bF=(\be, \bh, \bp, \dot{\bp}, \bm, \dot{\bm}) \in \mathcal{H}$ the system:
\begin{equation}\label{eq.systresolv}
(\bbA-\omega \, \mathrm{I}) \bU= \bF
\end{equation}
admits a unique solution $\bU =(\bE, \bH, \bP, \dot{\bP},\bM , \dot{\bM}) \in D(\bbA)$. To prove this, one first eliminates $\dot{\bP}$ and $\dot{\bM}$   in the system \eqref{eq.systresolv} by using that  $ \dot {\bP} =-\rmi \omega \, \bP -\rmi \, {\bp} $ and  $ \dot {\bM} =-\rmi \omega \, \bM -\rmi \, {\bm}$  and  obtain the following expression  for $\bP$ and 
$\bM$ in term of  $\bE$ and $\bH$:
\begin{subequations} 
\begin{empheq}[left=\empheqlbrace]{align}
 & {\bP}_{j}= - \frac{\bE}{q_{e,j}(\omega)}+ {\bF}_{p,j} (\omega) \quad  \mbox{ where }  \quad \bF_{p,j}(\omega)=\frac{( -\rmi \alpha_{e,j}-\omega ) \bp_ j-\rmi \dot{\bp}_j}{q_{e,j}(\omega)}, \label{eq.pj} \\[10pt]
&\bM_{\ell} = -\frac{\bH}{q_{m,l}(\omega)}+ \bF_{m,\ell} (\omega) \quad \mbox{ where } \quad  \bF_{m,l}(\omega)=\frac{( -\rmi \alpha_{m,\ell}-\omega ) \bm_\ell-\rmi \dot{\bm}_{\ell}}{q_{m,\ell}(\omega)} ,\label{eq.ml}
\end{empheq}
\end{subequations}
where 
and
$
q_{e,j}(\omega)=\omega^2+ i \alpha_{e,j} \omega- \omega_{e, j}^2\neq 0 \mbox{ and } q_{m,\ell}(\omega)=\omega^2+ i \alpha_{m, \ell} \omega- \omega_{m, \ell}^2\neq 0 \quad  \mbox{for $\omega\in \C^+$}.$
We point out that if $\omega_*$ is zero of $q_e$, $-\overline{\omega}_*$ is also zero of $q_e$, thus $\operatorname{Im}(\omega_*)=-\alpha_{e,j}/2\leq 0$ and therefore $\omega\in \C^+$ is not a zero of $q_{e,j}$. The same holds for the polynomials $q_{m,\ell}$.
Thus, it follows the expressions of $\dot{\bP}$ and $\dot{\bM}$ in terms of $\bE$ and $\bH$:
\begin{subequations} 
\begin{empheq}[left=\empheqlbrace]{align}
& \dot {\bP}_{j}= \frac{ \rmi \, \omega  \,\bE}{q_{e,j}(\omega)}+ \dot{\bF}_{p,j} (\omega) \quad  \ \ \mbox{ where }  \quad \dot{\bF}_{p,j} (\omega)=\frac{\rmi  \, \omega_{e,j}^2\,  \bp_ j  -\omega \dot{\bp}_{ j}}{q_{e,j}(\omega)}, \label{eq.dotpj} \\
&\dot {\bM}_{\ell} = \frac{ \rmi \, \omega  \,\bH}{q_{m,l}(\omega)}+ \dot{\bF}_{m,\ell} (\omega) \quad \mbox{ where } \quad  \dot{\bF}_{m,l}(\omega)=\frac{\rmi  \, \omega_{m,l}^2\,  \bm_{\ell }  -\omega\dot{\bm}_{\ell}}{q_{m,\ell}(\omega)}.\label{eq.dotml}
\end{empheq}
\end{subequations}
Eliminating $\dot{\bP}_j$ and $\dot{\bM}_{\ell}$ in the two first equations  of \eqref{eq.systresolv} yields to the following system in $\bE$ and $\bH$:
\begin{subequations} 
\begin{empheq}[left=\empheqlbrace]{align}
-&\nabla \times \bH- \rmi \omega \, \varepsilon (\omega) \, \bE= \bF_e(\omega) \quad  \mbox{ with } \quad \bF_e(\omega)= \varepsilon_0[ \rmi \,\be-\sum_{j=1}^{N_e}  \Omega_{e,j}^2 \dot{\bF}_{p,j} (\omega)],  \label{eq.e2} \\ 
&\nabla \times \bE - \rmi \omega  \mu(\omega) \, \bH= \bF_h(\omega) \quad  \mbox{ with } \quad  \bF_h(\omega)=  \mu_0\, [\rmi\, \bh-  \sum_{l=1}^{N_m}\Omega_{m,l}^2\dot{\bF}_{m,l} (\omega)]. \label{eq.h2}
\end{empheq}
\end{subequations}
where $\varepsilon(\omega)$ and $\mu(\omega)$ are defined by \eqref{epsmuLorentz}. Substituting 
\begin{equation}\label{eq.H}
\bH=\frac{\rmi}{\omega \mu(\omega)} \,\big(-\nabla \times \bE+\bF_h(\omega)\big)
\end{equation}
in \eqref{eq.e2} yields
\begin{equation}\label{eq.E}
-\frac{1}{\omega \mu(\omega)}\nabla\times (\nabla \times \bE)+\omega \, \varepsilon(\omega)\,\bE=-\frac{1}{\omega \mu(\omega)} \nabla\times  \bF_h(\omega)
+\rmi \, \bF_e(\omega).
\end{equation}
One shows using standard arguments that $\bE\in H(\mbox{rot};\R^3)$ is solution of \eqref{eq.E} if and only if $\bE$ is  solution to the following variational problem in $H_{\curl}(\R^3)$:
$$
a(\bE,\bpsi)=l(\bpsi), \quad \forall \bphi \in H(\mbox{rot};\R^3),
$$
where the sequilinear form $a$ and the  antilinear form $l$ are defined for all $\bphi, \bpsi \in H(\mbox{rot};\R^3) $:
\begin{equation*}\label{eq.lax-milgramm}
a(\bphi,\bpsi)=  \int_{\R^3} \Big( \! -\frac{(\nabla \times  \bphi)   \cdot \overline{ ( \nabla \times  \bpsi)}}{\omega \mu(\omega)}   + \omega  \varepsilon(\omega)   \bphi \cdot \overline{\psi}\, \Big) \,\mathrm{d}\bx,  \ l(\bpsi)= \int_{\R^3} \Big( -\frac{\bF_h(\omega)  }{\omega \mu(\omega)}  \cdot \overline{ ( \nabla \times  \bpsi)}+\rmi \, \bF_e(\omega) \cdot \overline{\psi}\Big)\, \mathrm{d}\bx.
\end{equation*}
By using the Cauchy-Schwarz inequality, it is clear that $a$  
(resp. $l$) is continuous on  $ H(\mbox{rot};\R^3)^2$ (resp. on $ H(\mbox{rot};\R^3)$). Furthermore, $\omega \mapsto \omega \epsilon(\omega) $ and $\omega  \mapsto \omega \mu(\omega)$ are Herglotz functions. Thus, $\omega \mapsto -1/(\omega \mu(\omega))$ is aslo an Herglotz function. Hence, one shows that for all $\omega\in \C^+$:
$$
|a(\bphi,\bphi)|\geq |\operatorname{Im}a(\bphi,\bphi)|\geq \gamma(\omega) \|\bphi\|_{H(\mbox{rot};\R^3)}^2 \mbox{ where } \alpha(\omega)= \min\Big( \operatorname{Im}\big(-(\omega \mu(\omega))^{-1} \big),\operatorname{Im} \big(\omega \varepsilon(\omega)\big)\Big)>0.
$$
Thus, by the Lax-Milgram theorem,  \eqref{eq.E} admits a unique solution in $H(\mbox{rot};\R^3)$. Then,   from \eqref{eq.H},    \eqref{eq.pj}, \eqref{eq.dotpj},  \eqref{eq.ml}, \eqref{eq.dotml},  $\bH, \bP, \dot{\bP}, \bM, \dot{\bM} $ are defined uniquely in term of $\bE$ as elements of $L^2(\R^3)^3$. Moreover, with \eqref{eq.e2}, $\nabla \times\bH\in L^2(\R^3)^3$, thus \eqref{eq.systresolv} admits a unique solution $\bU$ in $D(\bbA)$. 
\end{proof}
\noindent Under the weak dissipation assumption \eqref{WD}, $-\rmi \, \bbA$ is a maximal dissipative operator. Hence, $\bbA$ is a closed operator (see e.g. \cite{Dau-1992}, theorem 8 page 340) and its  spectrum $\sigma(\bbA)$  is contained in the closure of the lower half plane $\overline{\C^-}$.
Furthermore, it follows from the Lumer-Phillips theorem (see e.g. \cite{Dau-1992}, theorem 7 pages 336-337) that $-\rmi \bbA$ is the generator of a contraction semi-group $\{\mathcal{S}(t)\}_{t\geq 0}$ of class $\mathcal{C}^0$. 
This implies the following standard results  on the well-posdness of  \eqref{eq.schro} and the regularity of its solutions (where the Hilbert space $D(\bbA)$ is endowed with the graph norm defined by: $\|\bU\|_{D(\bbA)}^2=\|\bU\|^2_{\mathcal{H}}+\|\bbA \bU\|^2_{\mathcal{H}}, \, \forall \bU \in D(\bbA)$.

\begin{prop}\label{prop.wellposdness}
Let $m$ be an integer satisfying $m\geq 1$. If the initial condition $\bU_0\in D(\bbA^m)$ then the Cauchy problem 
\eqref{eq.schro} admits a unique strong solution $\bU \in C^{m}(\R^+, \mathcal{H})\cap  C^{m-1}(\R^+, D(\bbA))$ given by $\bU(t)=\mathcal{S}(t) \bU_0$ for $t\geq 0$. If $\bU_0\in\mathcal{H}$, 
\eqref{eq.schro} admits a unique mild  solution $\bU \in C^{0}(\R^+, \mathcal{H})$. 
\end{prop}
\noindent For the proof of this classical result,  we refer to \cite{Dau-1992} (theorem 1 page 399)  for the case $m=1$ and to  \cite{Bre-11} (Theorem 7.5 page 191) for the case $m>1$.  For the definition of mild solutions when the initial data $\bU_0$ is  in $\mathcal{H}$, we refer to \cite{Dau-1992} pages 404-405.
\begin{rem}
Under \eqref{WD}, $-\rmi \bbA$ is dissipative. 
The identity   $\rmd \mathcal{L}(t)/\rmd t=2\operatorname{Im}(\bbA \bU(t),\bU(t))_{\mathcal{H}}\leq 0$ with $
\mathcal{L}(t)=(\bU(t),\bU(t))_{\mathcal{H}}
$ (the abstract version of \eqref{decay_total_energy} or \eqref {decay_total_energyL}) implies the decay in time of $\mathcal{L}(t)$.
\end{rem}
\noindent We now define $\mathcal{H}_{\operatorname{div},0}$ the closed subspace of $\mathcal{H}$  given by
\begin{equation}
\mathcal{H}_{\operatorname{div},0}:=\{\bU=(\bE, \bH, \bP, \dot{\bP}, \bM, \dot{\bM})\in \mathcal{H}\mid \nabla \cdot \bE= \nabla \cdot \bH= \nabla \cdot \bP =\nabla \cdot \dot{\bP}=\nabla \cdot {\bM}=\nabla \cdot {\dot{\bM}}=0 \}.
\end{equation}
\noindent Finally, the following proposition explains that if the initial has divergence free, the fields  solutions to \eqref{eq.schro} remains divergence free for any $t\geq0$. 
\begin{prop}\label{prop.div}
Let $\bU_0\in D(\bbA) \cap  \mathcal{H}_{\operatorname{div},0}$ and $\bU:\R^+\to \mathcal{H}$ be  the unique strong solution of \eqref{eq.schro} with $\bU_0$
as initial condition then $\bU(t)\in D(\bbA) \cap  \mathcal{H}_{\operatorname{div},0}$ for any $t\geq 0$.
If $\bU_0\in \mathcal{H}_{\operatorname{div},0}$  then the mild solution $\bU$ of \eqref{eq.schro} satisfies also  $\bU(t)\in \mathcal{H}_{\operatorname{div},0}$ for any $t\geq 0$.

\end{prop}

\begin{proof}
We assume first that the initial condition  is regular enough, namely $\bU_0\in D(\bbA^2)$. 
Hence by Proposition \ref{prop.wellposdness},  the strong solution $\bU$ of \eqref{eq.schro} belongs to  $C^{2}(\R^+, \mathcal{H})\cap  C^{1}(\R^+, D(\bbA)) $ and thus the evolution  equation \eqref{eq.schro}  (or equivalently  the system of equations \eqref{planteamiento Lorentz}) holds also at $t=0$ and  one can differentiate it for $t\geq0$.
Taking the divergence in the distributional sens of  the equation \eqref{E Lorentz} for $t\geq 0$ and taking the divergence of the equations  $\partial_t(\eqref{P Lorentz})$ for $t\geq 0$ and $j \in \{ 1,\ldots, N_e\}$ leads to:
\begin{equation}\label{eq.div}
\partial_t\ \nabla \cdot \textbf{E}+\,\sum_{j=1}^{N_e}\,\Omega_{e,j}^2\,\partial_t\,\nabla \cdot \mathbf{P}_j=0 \   \mbox{ and } \  \partial_t^3\ \nabla \cdot \mathbf{P}_j+\alpha_{e,j}\,\partial_t^2\, \nabla \cdot \mathbf{P}_j+\omega_{e,j}^2\,\partial_t \nabla \cdot \mathbf{P}_j= \partial_t \nabla \cdot\mathbf{E}.
\end{equation}
Thus, substituting $\partial_t \nabla \cdot\mathbf{E}$ in  the second equation of \eqref{eq.div} gives 
\begin{equation}\label{eq.oscill}
 \partial_t^3\ \nabla \cdot \mathbf{P}_j+\alpha_{e,j}\,\partial_t^2\, \nabla \cdot \mathbf{P}_j+\omega_{e,j}^2\,\partial_t \nabla \cdot \mathbf{P}_j+\,\sum_{j=1}^{N_e}\,\Omega_{e,j}^2\,\partial_t\,\nabla \cdot \mathbf{P}_j=0 \mbox{ for }  j \in \{ 1,\ldots, N_e\} \mbox{  and $t\geq 0$},
\end{equation}
with initial conditions: $\nabla \cdot \mathbf{P}_j(\cdot, 0)= \partial_t \cdot \nabla \cdot \mathbf{P}_j(\cdot, 0)=0$ (since $\bU_0 \in \mathcal{H}_{\operatorname{div},0}$). Then taking the divergence of  equation  \eqref{P Lorentz} evaluated at $t=0$ yields $$ \partial_t ^2\cdot \nabla \cdot \mathbf{P}_j(\cdot, 0)=- \alpha_{e,j}\,\partial_t\, \nabla \cdot \mathbf{P}_j(\cdot, 0)-\omega_{e,j}^2\,\partial_t \nabla \cdot \mathbf{P}_j(\cdot, 0)+ \partial_t \nabla \cdot\mathbf{E}(\cdot, 0)=0.$$
Thus, from \eqref{eq.oscill}, one deduces  $\nabla \cdot \mathbf{P}_j(\cdot,t)=\nabla \cdot \partial_t \mathbf{P}_j(\cdot,t)=0$ for any $t\geq 0$ and  $j \in \{ 1,\ldots, N_e\}$. It follows  with \eqref{eq.div} that $\partial_t \nabla \cdot\mathbf{E}(\cdot, t)=0$ with $  \nabla \cdot\mathbf{E}(\cdot, 0)=0$. Therefore, one gets also $\nabla \cdot\mathbf{E}(\cdot, t)=0$.
Similarly using  \eqref{H Lorentz} and \eqref{M Lorentz} and their derivatives, one shows that: 
$\nabla \cdot \mathbf{M}_\ell(\cdot,t)=\nabla \cdot \partial_t \mathbf{M}_\ell(\cdot,t)=\nabla \cdot\mathbf{\bH}(\cdot, t)= 0$  $t\geq 0$ and  $\ell \in \{ 1,\ldots, N_e\}$ and concludes that $\bU(t)\in  D(\bbA) \cap  \mathcal{H}_{\operatorname{div},0}$ for any $t\geq 0$.
\\ [12pt]
If $\bU_0$ is less regular, i.e. $\bU_0$ in $\mathcal{H}_{\operatorname{div},0}$, one obtains that  $\bU(t)=S(t)\bU_0\in  \mathcal{H}_{\operatorname{div},0}$ from the previous reasoning  by using  a density argument on the initial condition and the fact that the elements $S(t)$ of the $C^{0}$ semi-group corresponding  to \eqref{eq.schro} are contractions. Moreover, if $\bU_0\in D(\bbA)$, one has, by  Proposition \ref{prop.wellposdness}, $\bU(t)\in  D(\bbA)$  and thus $\bU(t)\in D(\bbA) \cap \mathcal{H}_{\operatorname{div},0}$.
%
\end{proof}

\begin{rem}\label{rem.Drude}
The (Cauchy) problem \eqref{planteamiento Lorentz}  assumes that all  the resonance frequencies satisfy $\omega_{e,j},  \omega_{m,\ell} > 0$, in other words that  $\varepsilon(\omega)$ and $\mu(\omega)$ only contained Lorentz terms (in the sense of Section \ref{Drude_Lorentz}). However, all the results in this appendix still hold with a  similar proof  if they contain Drude terms. 
In that case,  if  $N_{d,e}$ and $N_{d,m}$  are the number of electric and magnetic Drude terms, one has  to redefine  $\mathcal{H}$ in \eqref{eqdefHilbertspace}
 as $$\mathcal{H}:=L^2(\mathbb{R}^3)^3\times L^2(\mathbb{R}^3)^3\times L^2(\mathbb{R}^3)^{3(N_e - N_{d,e})}
 \times L^2(\mathbb{R}^3)^{3N_e} \times L^2(\mathbb{R}^3)^{3(N_m - N_{d,m})}  \times L^2(\mathbb{R}^3)^{3N_m} $$ ($\bP_j$ (resp. $\bM_{\ell}$) are no longer unknowns of the evolution problem for Drude terms). 
\end{rem}

\end{document}